\documentclass[11pt, a4paper]{amsart}
\usepackage{color}
\usepackage{graphicx}
\usepackage[english]{babel}
\usepackage{amssymb}
\usepackage{amscd}
\usepackage{amsmath}
\usepackage{mathrsfs}
\def\RR{ I \! \! R}

\def\di{\displaystyle}

\def\R{{\mathbb {R}}}

\def\uep{u^{\ep}}
\def\uepj{u^{\ep_j}}
\def\uepk{u^{\ep_k}}
\def\vep{v^{\ep}}

\def\H{{\mathcal{H}}^{N-1}}

\def\L{{\mathcal{L}}}
\def\pint{\operatorname {--\!\!\!\!\!\int\!\!\!\!\!--}}

\def\ep{\varepsilon}

\def\wg{W^{1,G}(\Omega)}
\def\wg0{W_0^{1,G}(\Omega)}

\newtheorem{teo}{Theorem}[section]
\newtheorem{lema}{Lemma}[section]
\newtheorem{defi}{Definition}[section]
\newtheorem{prop}{Proposition}[section]
\newtheorem{corol}{Corollary}[section]
\theoremstyle{definition}
\newtheorem{remark}{Remark}[section]

\renewcommand{\theequation}{\arabic{section}.\arabic{equation}}

\begin{document}
\title[Singular perturbation for a degenerate or singular quasilinear operator]
{A singular perturbation problem for a quasilinear operator
satisfying the natural growth condition of Lieberman}

\author[ S. Mart\'{i}nez \& N. Wolanski]
{Sandra Mart\'{\i}nez and Noemi Wolanski}

\address{Departamento  de Matem\'atica, FCEyN
\hfill\break\indent UBA (1428) Buenos Aires, Argentina.}
\email{{ smartin@dm.uba.ar \\
wolanski@dm.uba.ar} \hfill\break\indent {\em Web-page:}
 {\tt http://mate.dm.uba.ar/$\sim$wolanski}}

\thanks{Supported by ANPCyT PICT No.
03-13719, UBA X52 and X66 and CONICET PIP 5478.}
\thanks{N.
Wolanski is a member of CONICET}



\maketitle

\begin{abstract}
In this paper we study the following problem. For any  $\ep>0$,
take $u^{\ep}$ a solution of,
$$
 \L u^{\ep}:= \mbox{div\,}\Big(\di\frac {g(|\nabla \uep|)}{|\nabla
\uep|}\nabla \uep\Big)=\beta_{\ep}(u^{\ep}),\quad u^{\ep}\geq 0.
$$

A solution to  $(P_{\ep})$ is a function $u^{\ep}\in
W^{1,G}(\Omega)\cap L^{\infty}(\Omega)$ such that
$$
\int_{\Omega} g(|\nabla u^{\ep}|) \frac{\nabla u^{\ep}}{|\nabla
u^{\ep}|} \nabla \varphi\, dx =-\int_{\Omega} \varphi\,
\beta_{\ep}(u^{\ep})\, dx
$$
for every  $\varphi \in C_0^{\infty}(\Omega)$.

Here $\beta_{\ep}(s)= \frac{1}{\ep}
\beta\left(\frac{s}{\ep}\right), $ with  $\beta\in
\mbox{Lip}(\R)$, $\beta>0$ in $(0,1)$ and $\beta=0$ otherwise.

We are interested in  the limiting problem, when $\ep\to 0$. As in
previous work with $\L=\Delta$ or $\L=\Delta_p$ we prove, under
appropriate assumptions, that any limiting function is a weak
solution to a free boundary problem. Moreover, for nondegenerate
limits we prove that the reduced free boundary is a $C^{1,\alpha}$
surface. This result is new even for $\Delta_p$.

Throughout the paper we  assume that $g$ satisfies the conditions
introduced by G. Lieberman in \cite{Li1}.
\end{abstract}

\section{Introduction}
In this paper we study, the following singular perturbation
problem: For any $\ep>0$, take $u^{\ep}$ a nonnegative solution
of,
\begin{equation*}\label{ecsing}
\tag{$P_\ep$} \L u^{\ep}=\beta_{\ep}(u^{\ep}),\quad u^{\ep}\geq 0,
\end{equation*}
where $\L v:= \mbox{div\,}\Big(\di\frac {g(|\nabla v|)}{|\nabla
v|}\nabla v\Big)$.

 A solution to $(P_{\ep})$  is a function
$u^{\ep}\in W^{1,G}(\Omega)\cap L^{\infty}(\Omega)$ (see the
notation for the definition of $W^{1,G}(\Omega)$) such that
\begin{equation}\label{ecsin}
\int_{\Omega} g(|\nabla u^{\ep}|) \frac{\nabla u^{\ep}}{|\nabla
u^{\ep}|} \nabla \varphi\, dx =-\int_{\Omega} \varphi\,
\beta_{\ep}(u^{\ep})\, dx
\end{equation}
 for every $\varphi \in
C_0^{\infty}(\Omega)$.

Here  $\beta_{\ep}(s)= \frac{1}{\ep}
\beta\left(\frac{s}{\ep}\right), $ for $\beta\in \mbox{Lip}(\R)$,
 positive in $(0,1)$ and zero otherwise. We call  $M= \int_0^{1} \beta(s)\, ds$.

We are interested in studying the uniform properties of solutions
and understanding what happens in the limit as $\ep\to 0$. We
assume throughout the paper that the family $\{\uep\}$ is
uniformly bounded in $L^\infty$ norm.
 Our aim is to prove that, for every sequence $\ep_n\to 0$
there exists a subsequence $\ep_{n_k}$ and a function
 $u=\lim u^{\ep_{n_k}}$, and that $u$ is a weak
solution of the free boundary problem
\begin{equation}\label{bernoulli2}
\begin{cases}
\L u:= \mbox{div\,}\Big(\di\frac {g(|\nabla u|)}{|\nabla u|}\nabla u\Big)=0\quad&\mbox{in }\{u>0\}\cap\Omega\\
|\nabla u|=\lambda^*\quad&\mbox{on }\partial\{u>0\}\cap\Omega.\\
\end{cases}
\end{equation}
 for some constant
$\lambda^*$ depending on $g$ and $M$.

\bigskip

This problem appears in combustion theory in the case $\L=\Delta$
when studying deflagration flames.
  Back in 1938, Zeldovich and  Frank-Kamenetski  proposed the passage to the limit in this singular perturbation
  problem in \cite{ZF} (the limit for the activation energy going to infinity
    in this flame propagation model). The passage to the limit was not studied in a  mathematically rigorous way until
   1990 when
  Berestycki, Caffarelli and Nirenberg studied the case of $N$ dimensional traveling waves  (see
 \cite{BCN}). Later, in   \cite{CV},  the general evolution problem in the one phase case was considered. Much research
 has been done on this matter ever since. (See, for instance, \cite{CLW1, CLW2, LW1, W}).

 \eqref{bernoulli2} is a very well known
free boundary problem in the uniformly elliptic case ($0<c\le
g(t)/t\le C<\infty$). This problem  has also been studied in the
two phase case. Regularity results for the free boundary in the
case of the laplacian can be found in \cite{AC} for one phase
distributional solutions  and in \cite{C1, C2} for two phase
viscosity solutions. See also \cite{ACF} for one phase
distributional solutions in the nonlinear uniformly elliptic case.
The results in \cite{AC,C1,C2} were used in \cite{LW1} to obtain
free boundary regularity results for limit solutions (this is, for
$u=\lim\uepk$). See also \cite{CJK,LW2} for results in the
inhomogeneous case and \cite{CFS, FS} for viscosity solutions in
the variable coefficient case.

Recently, this singular perturbation problem in the case of the
$p$--laplacian ($g(t)=t^{p-1}$)
was considered in  \cite{DPS}. 
 As in the uniformly elliptic case, the authors find, for
a uniformly bounded family of solutions $u^{\ep}$, Lipschitz
estimates uniform in $\ep$  and prove that the limit of $u^{\ep}$
is a  solution of  \eqref{bernoulli2} for $\L=\Delta_p$ and
$\lambda^*=\big(\frac{p}{p-1} M\big)^{1/p}$ in a pointwise sense
at points in the reduced free boundary.



\bigskip

The aim of our present work is to study   this singular
perturbation problem --including the regularity of the free
boundary-- for operators that can be elliptic degenerate or
singular, possibly non homogeneous (the $p$--laplacian is
homogeneous and this fact simplifies some of the proofs).
Moreover, we admit functions  $g$ in the operator
 $\L$ with a different behavior at  0 and at infinity.
 Classically, the assumptions on the behavior of  $g$ at
0 and at infinity were  similar to the case of the $p$--laplacian.
Here, instead, we adopt the conditions introduced by G. Lieberman
in \cite{Li1} for the study of the regularity of weak solutions of
the elliptic equation
 (possibly degenerate or singular) $\L u=f$ with  $f$ bounded.

This condition ensures that the equation  $\L u=0$ is equivalent
to a uniformly elliptic equation in nondivergence form with
constants of ellipticity independent of the solution  $u$ in sets
where  $\nabla u\neq0$. Furthermore, this condition does not imply
any type of homogeneity on the function   $g$ and, moreover it
allows for a different behavior of  $g(|\nabla u|)$ when $|\nabla
u|$ is near zero or infinity. Precisely, we assume that  $g$
satisfies
\begin{equation}\label{cond}
0<\delta\le\frac{t g'(t)}{g(t)}\le g_0\ \ \ \ \forall t>0
\end{equation}
for certain constants  $0<\delta\le g_0$.

Let us observe that  $\delta=g_0=p-1$ when $g(t)=t^{p-1}$, and
reciprocally, if $\delta=g_0$ then  $g$ is a power.

Another example of a function   $g$ that satisfies \eqref{cond} is
the function
 $g(t)=t^a\mbox{log\,}(b t+c)$ with  $a, b, c>0$. In this case,
 \eqref{cond} is satisfied with   $\delta=a$ and $g_0=a+1$.

Another interesting case is the one of functions $g\in
C^1([0,\infty))$ with $g(t)=c_1t^{a_1}$ for  $t\le s$,
$g(t)=c_2t^{a_2}+d$ for $t\ge s$. In this case   $g$ satisfies
 \eqref{cond} with  $\delta=\mbox{min}(a_1,a_2)$ and
$g_0=\mbox{max}(a_1,a_2)$.

Furthermore, any linear combination with positive coefficients of
functions satisfying  \eqref{cond} also satisfies \eqref{cond}. On
the other hand, if  $g_1$ and $g_2$ satisfy  \eqref{cond} with
constants $\delta^i$ and $g_0^i$, $i=1,2$, the function $g=g_1g_2$
satisfies \eqref{cond} with $\delta=\delta^1+\delta^2$ and
$g_0=g_0^1+g_0^2$, and the function $g(t)=g_1\big(g_2(t)\big)$
satisfies \eqref{cond} with $\delta=\delta^1\delta^2$ and
$g_0=g_0^1g_0^2$.

This observation shows that there is a wide range of functions $g$
under the hypothesis of this work.

\bigskip

In this paper we show in this paper that the limit functions are
solutions of \eqref{bernoulli2} in the weak sense introduced in
\cite{MW1} where we proved that the reduced boundary of these weak
solutions is a $C^{1,\alpha}$ surface. This notion of weak
solution turns out to be very well suited for limit functions of
this singular perturbation problem.

\bigskip

We state here the definition of weak solution and the main results
in this paper.

\begin{defi}[Weak solution II in \cite{MW1}]\label{weak2}
We call $u$ a weak solution of \eqref{bernoulli2} if
\begin{enumerate}
\item $u$ is continuous and non-negative in $\Omega$ and
$\mathcal{L}u=0$ in $\Omega\cap\{u>0\}$. \item For
$D\subset\subset \Omega$ there are constants $0< c_{min}\leq
C_{max}$, $\gamma\ge1$, such that for balls $B_r(x)\subset D$ with
$x\in \partial \{u>0\}$
$$
c_{min}\leq \frac{1}{r}\Big(\pint_{B_r(x)} u^{\gamma} dx\,
\Big)^{1/\gamma}\leq C_{max}
$$
\item For $\mathcal{H}^{N-1}$ a.e $x_0\in\partial_{{red}}\{u>0\}$,
$u$ has the asymptotic development
$$
u(x)=\lambda^*\langle x-x_0,\nu(x_0)\rangle^-+o(|x-x_0|)
$$
where $\nu(x_0)$ is the unit interior normal to $\partial\{u>0\}$
at $x_0$ in the measure theoretic sense.

\item For every $ x_0\in \Omega\cap\partial\{u>0\}$,
\begin{align*}
& \limsup_{\stackrel{x\to x_0}{u(x)>0}} |\nabla u(x)| \leq
\lambda^*.
\end{align*}

If there is a ball $B\subset\{u=0\}$ touching
$\Omega\cap\partial\{u>0\}$ at $x_0$ then,
$$\limsup_{\stackrel{x\to x_0}{u(x)>0}} \frac{u(x)}{\mbox{dist}(x,B)}\geq  \lambda^*. $$
\end{enumerate}
\end{defi}


\bigskip

Our first result is a bound of $\|\nabla \uep\|_{L^\infty}$
independent of $\ep$.

\begin{teo}
Let $u^{\ep}$ be a solution of
$$\L u^{\ep}=\beta_{\ep}(u^{\ep}) \quad \mbox{ in } \Omega,$$ with $\|u^{\ep}\|_{L^{\infty}(\Omega)}\leq
L$.
  Then, for $\Omega'\subset\subset \Omega$ we have,
$$|\nabla u^{\ep}(x)|\leq C\quad \mbox{ in } \Omega'$$
with
$C=C(N,\delta,g_0,L,\|\beta\|_{\infty},g(1),\mbox{dist}(\Omega',\partial\Omega))$,
if $\ep\leq \ep_0(\Omega,\Omega')$.
\end{teo}

Then we have, via a subsequence, that there exists a limiting
function $u$.

The next step is to prove that the function  $u$ is a weak
solution in the sense of Definition \ref{weak2} of the free
boundary problem \eqref{bernoulli2} for a constant $\lambda^*$
depending on $g$ and $M$. To this end, we have to prove that $\L
u=0$ in $\{u>0\}$ and that we have an asymptotic development for
$u$ at any point on the reduced free boundary.

Here we find several technical difficulties associated to the loss
of homogeneity of the operator $\L$ and to the fact that we are
working in an Orlicz space. This is the case, for instance when we
need to prove the pointwise convergence of the gradients.


At some point  we need  to add the following hypothesis on   $g$:

{ There exists $\eta_0>0$ such that,
\begin{equation}\label{condderivadagintro2}
g'(t)\leq s^2 g'(ts)  \quad \mbox{if}\quad 1\leq s\leq 1+\eta_0
\quad \mbox{ and }\quad 0<t\le
\Phi^{-1}\Big(\frac{g_0}{\delta}M\Big),
\end{equation}}
where $\Phi(\lambda)=\lambda g(\lambda)-G(\lambda)$.

We remark that condition \eqref{condderivadagintro2} holds for all
the examples of functions satisfying condition \eqref{cond}
described above (see Section 4).

%

There holds,

\begin{teo}
Suppose that $g$ satisfies \eqref{cond} and
\eqref{condderivadagintro2}. Let $u^{\ep_j}$ be a solution to
$(P_{\ep_j})$ in a domain $\Omega\subset \mathbb{R}^N$ such that
$u^{\ep_j}\rightarrow u$ uniformly on compact subsets of $\Omega$
and $\ep_j\to 0$. Let $x_0\in \Omega \cap\partial\{u>0\}$ be such
that $\partial\{u>0\}$ has an inward unit normal $\nu$ in the
measure theoretic sense at $x_0$, and suppose that $u$ is
non-degenerate at $x_0$ (see Definition  \ref{nondegener}). Under
these assumptions, we have
$$u(x)=\Phi^{-1}(M) \langle x-x_0,\nu\rangle^+ +o(|x-x_0|)$$
where $\Phi(\lambda)=\lambda g(\lambda)-G(\lambda)$.
\end{teo}

 Finally, we can apply the theory developed in
\cite{MW1}. We have that $u$ is a weak solution in the sense of
Definition \ref{weak2} of the free boundary problem.

Then, we have the following,

\begin{teo}\label{teointro}
Suppose that  ${g}$ satisfies \eqref{cond} and
\eqref{condderivadagintro2}. Let $u^{\ep_j}$ be a solution of
$(P_{\ep_j})$ in a domain $\Omega\subset\mathbb{R}^N$ such that
$u^{\ep_j}\rightarrow u$ uniformly in compact subsets of $\Omega$
as $\ep_j\to 0$. Let $x_0\in \Omega\cap\partial\{u>0\}$ such that
there is a unit inward normal   $\nu$ to
$\Omega\cap\partial\{u>0\}$ in the measure theoretic sense at
$x_0$. Suppose that
 $u$ is uniformly non-degenerate at the free boundary in a neighborhood of $x_0$
 (see Definition  \ref{nondegener}).
Then, there exists   $r>0$ such that $B_r(x_0)\cap\partial\{u>0\}$
is a  $C^{1,\alpha}$ surface.
\end{teo}

{ Finally, we give two examples in which we can apply the
regularity results in this paper. In both examples  the
nondegeneracy property  is satisfied by the limiting function $u$.
In the first example the limiting function is obtained by taking a
sequence of minimal solutions of $(P_{\ep})$ (see Definition
\ref{minimal}) . In the second one, by taking a sequence of
minimizers of the functional
$$J_{\ep}(v)= \int_{\Omega} [G(|\nabla v|) +  B_{\ep}( v)]\,
dx$$ where $B_{\ep}'(s)=\beta_{\ep}(s)$
 (see Section 7).

 Moreover, in the second example
we have that $\H(\partial\{ u>0\}\setminus
\partial_{red}\{u>0\})=0$. Thus, in this case the set of singular points has zero $\H-$measure.


 We
also have --since the limiting function is a minimizer of the
problem considered in \cite{MW1}-- that in the case of minimizers
we don't need to add any new hypothesis  to the function $g$. This
is, the result holds for functions $g$ satisfying only condition
\eqref{cond}. And, in dimension 2 if we add to condition
\eqref{cond}  that,
\begin{equation}\label{condising}
\mbox{There exist  constants }  t_0>0 \mbox{ and } k>0 \mbox{ so
that } g(t)\leq k t \mbox{ for } t\leq t_0.
\end{equation}
then, we have that the whole free boundary is a regular surface
(see Corollary 2.2 in \cite{M1}).

\bigskip}

\subsection*{Outline of the paper}

The paper is organized as follows: In Section 3 we prove the
uniform Lipschitz continuity of solutions of \eqref{ecsing}
(Corollary \ref{gradsing}).

 In
Section 4 we prove that if $u$ is a limiting function, then $\L u$
is a Radon measure supported on the free boundary (Theorem
\ref{nose}). Then we prove
 Proposition \ref{alphamenorphiM}, that says that if $u$
is a half plane, then the slope is 0 or  $\Phi^{-1}(M)$, and
Proposition \ref{gammapositiva} that says  that if $u$ is a sum of
two half planes, then the slopes must be equal and at most
$\Phi^{-1}(M)$.

In Section 5 we prove the  asymptotic development of $u$ at points
in the reduced free boundary (Theorem \ref{teoasim}) and we prove
that $u$ is a weak solution according to Definition \ref{weak2}.

 In section 6 we apply the results of \cite{MW1} to prove the regularity of the free
boundary (Theorem \ref{regularidadsing}).

{In Section 7 we give two examples where the limiting function
satisfies the nondegeneracy property. The first one is given by
the limit of minimal solutions (Theorem \ref{ejemplo1}) and the
second one is given by the limit of energy minimizers (Theorem
\ref{ejemplo2}).}

In the Appendices we state some properties of the function $g$ and
we  prove the asymptotic development of $\L$--subsolutions.
%
%
%
%

\section{Notation}

Throughout the paper $N$ will denote the dimension and,
\begin{align*}B_r(x)&=\{ x\in \R^N, |x-x_0|<r\},\\
B^+_r(x)&=\{ x\in \R^N, x_N> 0,\ |x-x_0|<r\},\\
B^-_r(x)&=\{ x\in \R^N, x_N< 0,\ |x-x_0|<r\}.
\end{align*}

For $v,w\in \R^N$, $\langle v,w\rangle$ denotes the standard
scalar product.

For a scalar function $f$, $f^+=\max({f,0})$ and
$f^-=\max({-f,0})$.

Furthermore, we denote
\begin{align*} G(t)&=\int_0^t g(s)\,ds,\\
F(t)&=g(t)/t,\\ \Phi(t)&=g(t)t-G(t),
\\  A(p)&=F(|p|) p\quad \mbox{ for } p\in \R^N,\\
a_{ij}&=\frac{\partial A_i}{\partial p_j}\quad \ \ \ \ \mbox{ for
} 1\leq i,j\leq N.
\end{align*}

We denote by $L^G(\Omega)$ the  Orlicz space that is the linear
hull of the set of measurable functions such that $\int_\Omega
G(|u|)\,dx<\infty$ with the norm of Luxemburg. This is,
$$
\|u\|_{L^G(\Omega)}=\inf\Big\{\lambda>0\,/\,\int_\Omega
G\Big(\frac{|u|}\lambda\Big)\,dx\le 1\Big\}.
$$

 The set $W^{1,G}(\Omega)$ is the Sobolev Orlicz space of functions in $W^{1,1}_{loc}(\Omega)$ such that
 both $\|u\|_{L^G(\Omega)}$ and $\||\nabla u|\|_{L^G(\Omega)}$ are finite equipped, with the norm
 $$
\|u\|_{W^{1,G}(\Omega)}=
\mbox{max\,}\big\{\|u\|_{L^G(\Omega)},\||\nabla
u|\|_{L^G(\Omega)}\big\}.
$$

\bigskip

\section{Uniform bound of the gradient}
\setcounter{equation}{0}

We begin by proving that solutions of the perturbation problem are
locally uniformly Lipschitz. That is, the $u^{\ep}$'s are locally
Lipschitz, and the Lipschitz constant is bounded independently of
$\ep$. In order to prove this result, we first need to prove a
couple of lemmas.

\begin{lema}\label{umenor2ep}
Let $u^{\ep}$ be a solution of
$$\L u^{\ep}=\beta_{\ep}(u^{\ep})\quad \mbox{ in } B_{r_0}(x_0)$$
such that $u^{\ep}(x_0)\leq 2\ep$. Then, there exists
$C=C(N,r_0,\delta,g_0,\|\beta\|_{\infty},g(1))$ such that, if
$\ep\leq 1$,
$$|\nabla u^{\ep}(x_0)|\leq C.$$
\end{lema}
\begin{proof}
Let $v(x)=\frac{1}{\ep} u^{\ep}(x_0+\ep x)$. Then if $\ep\leq1$,
$\L v=\beta(v)$ in $B_{r_0}$ and $v(0)\leq 2$. By Harnack's
inequality (see \cite{Li1}) we have that $0\leq v(x)\leq C_1$ in
$B_{r_0/2}$ with $C_1=C_1(N,g_0,\delta,\|\beta\|_{\infty})$.
Therefore, by using the derivative estimates of \cite{Li1}  we
have that
$$|\nabla \uep(x_0)|=|\nabla v(0)|\leq C$$ with $C=C(N,\delta,g_0,\|\beta\|_{\infty},r_0,g(1)).$
\end{proof}

\begin{lema}\label{uep=ep}
Let $u^{\ep}$ be a solution of
$$\L u^{\ep}=\beta_{\ep}(u^{\ep}) \quad \mbox{ in } B_{1},$$
and $0\in\partial\{u^{\ep}>\ep\}$. Then, for $x\in
B_{1/4}\cap\{u^\ep>\ep\}$,
$$u^{\ep}(x)\leq \ep+C\mbox{dist}(x,\{u^{\ep}\leq \ep\}\cap B_1),$$ with $C=C(N,\delta,g_0,\|\beta\|_{\infty},g(1))$.
\end{lema}
\begin{proof}
For $x_0\in B_{1/4}\cap\{u^\ep>\ep\}$ take, $m_0=u^{\ep}(x_0)-\ep$
and $\delta_0=\mbox{dist}(x_0,\{u^{\ep}\leq\ep\}\cap B_1)$. Since
$0\in \partial \{u^{\ep}>\ep\}\cap B_1$, $\delta_0\leq 1/4$. We
want to prove that, $m_0\leq
C(N,\delta,g_0,\|\beta\|_{\infty},g(1))\delta_0$.

Since, $B_{\delta_0}(x_0)\subset \{u^{\ep}>\ep\}\cap B_1$ we have
that, $u^{\ep}-\ep> 0$ in $B_{\delta_0}(x_0)$ and
$\L(u^{\ep}-\ep)=0$. By Harnack's inequality there exists
$c_1=c_1(N,g_0,\delta)$ such that
$$\min_{B_{\delta_0/2}(x_0)}(u^{\ep}-\ep)\geq c_1 m_0.$$
Let us take $\varphi=e^{-\mu|x|^2}-e^{-\mu \delta_0^2}$ with
$\mu={2K}\,/\,{\delta\delta_0^2}$,  where  $K=2N$ if $g_0<1$ and
$K=2(g_0-1)+2N$ if $g_0\ge 1$. Then, we have that $\L \varphi
>0$ in $B_{\delta_0}\setminus B_{{\delta_0/2}}$ (see the proof of  Lemma 2.9 in \cite{MW1}).

Let now $\psi(x)=c_2 m_0 \varphi\left({x-x_0}\right)$ for $x\in
\overline{B_{\delta_0}(x_0)}\setminus {B_{\delta_0/2}(x_0)}$.
Then, again by Lemma 2.9 in \cite{MW1}, we have that, if we choose
$c_2$ conveniently depending on $N, \delta$, $g_0$,
$$\begin{cases} \L \psi (x)>0 &\quad \mbox{ in } B_{\delta_0}(x_0)\setminus \overline{B_{\delta_0/2}(x_0)}\\
\psi =0 & \quad \mbox{ on }\partial B_{\delta_0}(x_0)\\
\psi=c_1 m_0 &\quad \mbox{ on }\partial B_{\delta_0/2}(x_0).
\end{cases}
$$

By the comparison principle (see Lemma 2.8 in \cite{MW1}) we have,
\begin{equation}\label{compphiuep} \psi(x)\leq u^{\ep}(x)-\ep \quad \mbox{ in }
\overline{B_{\delta_0}(x_0)}|\setminus
{B}_{\delta_0/2}(x_0).\end{equation} Take $y_0\in
\partial {B}_{\delta_0}(x_0)\cap \partial \{u^{\ep}>\ep\}$. Then,
$y_0\in \overline{B_{1/2}}$ and
\begin{equation}\label{psicero}
\psi(y_0)=u^{\ep}(y_0)-\ep=0.
\end{equation}
Let $v^{\ep}=\frac1\ep u^{\ep}(y_0+\ep x)$. Then if $\ep<1$
 we have that $\L v^{\ep}=\beta(v^ {\ep})$ in $B_{1/2}$ and
 $v^{\ep}(0)=1$.
Therefore, by Harnack's inequality (see \cite{Li1}) we have that
$\max_{\overline{B}_{1/4}} v^{\ep}\leq \widetilde{c}$ and
\begin{equation}\label{graduv14}
|\nabla u^{\ep}(y_0)|=|\nabla v^{\ep}(0)|\leq \widetilde{c}
\max_{\overline{B}_{1/4}} v^{\ep}\leq c_3.
\end{equation}
Finally, by \eqref{compphiuep}, \eqref{psicero} and
\eqref{graduv14} we have that, $|\nabla \psi(y_0)|\leq |\nabla
u^{\ep}(y_0)|\leq c_3$.
 Observe that $|\nabla \psi(y_0)|=
c_2 m_0 e^{-\mu \delta_0^2} 2\mu \delta_0\leq c_3$. Therefore,
$$m_0\leq \frac{c_3 e^{\mu\delta_0^2}}{ c_2 2\mu \delta_0}=
\frac{c_3 \delta e^{2K/\delta}}{ c_2 4K}\delta_0
$$ and the result follows.

\end{proof}

Now, we can prove the main result of this section,

\begin{prop}\label{cotagradsing}
Let $u^{\ep}$ be a solution of $\L u^{\ep}=\beta_{\ep}(u^{\ep})
\mbox{ in } B_1$. Assume that  $0\in \partial\{u^{\ep}>\ep\}$.
Then, we have for $x\in B_{1/8}$,
$$|\nabla u^{\ep}(x)|\leq C$$ with $C=C(N,\delta,g_0,\|\beta\|_{\infty},g(1))$.

\end{prop}

\begin{proof}
By Lemma \ref{umenor2ep} we know that if $x_0\in \{u^{\ep}\leq
2\ep\}\cap B_{3/4}$ then,
$$|\nabla u^{\ep}(x_0)|\leq C_0$$ with
$C_0=C_0(N,\delta,g_0,\|\beta\|_{\infty},g(1))$.

 Let
$x_0\in B_{1/8}\cap \{u^{\ep}>\ep\}$ and
$\delta_0=\mbox{dist}(x_0,\{u^{\ep}\leq \ep\})$.

As $0\in\partial\{u^{\ep}>\ep\}$ we have that $\delta_0\leq 1/8$.
Therefore, $B_{\delta_0}(x_0)\subset \{u^{\ep}>\ep\}\cap B_{1/4}$
and then $\L u^{\ep}=0$ in $B_{\delta_0}(x_0)$ and,  by Lemma
\ref{uep=ep},
\begin{equation}\label{ecuuepig}
u^{\ep}(x)\leq \ep+ C_1 \mbox{dist}(x,\{u^{\ep}\leq \ep\}) \quad
\mbox{ in } B_{\delta_0}(x_0).
\end{equation}
\begin{enumerate}
\item

Suppose that $\ep<\bar{c} \delta_0$ with $\bar{c}$ to be
determined.  Let $v(x)=\frac1{\delta_0}{u^{\ep}(x_0+\delta_0 x)}$.
Then, $\mathcal{L}v=\delta_0\beta_{\ep} (u^{\ep}(x_0+\delta_0
x))=0$ in $B_1$. Therefore, by the results of \cite{Li1}
$$|\nabla v(0)|\leq \widetilde{C}\sup_{B_1} v,  $$ with $\widetilde{C}=\widetilde{C}(N,g_0,\delta,
g(1))$.  We obtain,
$$|\nabla u^{\ep}(x_0)|\leq \frac{\widetilde{C}}{\delta_0}
\sup_{B_{\delta_0}(x_0)} u^{\ep} \leq
\frac{\widetilde{C}}{\delta_0} (\ep+C\delta_0)\leq
\widetilde{C}(\bar{c}+C).$$

\item Suppose that $\ep\geq \bar{c} \delta_0$. By \eqref{ecuuepig}
we have,
$$u^{\ep}(x_0)\leq \ep+ C_1 \delta_0\leq \Big (1+\frac{C_1}{\bar{c}}\Big)\ep< 2 \ep,$$
if we choose $\bar{c}$ big enough. By Lemma \ref{umenor2ep}, we
have $|\nabla u^{\ep}(x_0)|\leq C$, with\\
$C=C(N,g_0,\delta,\|\beta\|_{\infty},g(1))$.
\end{enumerate}
The result follows.
\end{proof}

With these lemmas  we obtain the following,

\begin{corol}\label{gradsing}
Let $u^{\ep}$ be a solution of
$$\L u^{\ep}=\beta_{\ep}(u^{\ep}) \quad \mbox{ in } \Omega,$$ with $\|u^{\ep}\|_{L^{\infty}(\Omega)}\leq
L$.
  Then, we have for $\Omega'\subset\subset \Omega$, that there
  exists $\ep_0(\Omega,\Omega')$ such that if $\ep\leq
  \ep_0(\Omega,\Omega')$,
$$|\nabla u^{\ep}(x)|\leq C\quad \mbox{ in } \Omega'$$
with
$C=C(N,\delta,g_0,L,\|\beta\|_{\infty},g(1),\mbox{dist}(\Omega',\partial\Omega))$.
\end{corol}
\begin{proof}
Let $\tau>0$ such that $\forall x\in \Omega'$,
$\overline{B_{\tau}(x)}\subset \Omega$ and $\ep\leq \tau$. Let
$x_0\in \Omega'$.
\begin{enumerate}
\item If $\delta_0=\mbox{dist} (x_0,\partial\{u^{\ep}>\ep\})\leq
\tau/8$, let $y_0\in\partial\{u^{\ep}>\ep\}$ such that
$|x_0-y_0|=\delta_0$. Let $v(x)=\frac1\tau{u^{\ep}(y_0+\tau x)}$,
and $\bar{x}=\frac{x_0-y_0}{\tau},$ then $|\bar{x}|<1/8$. As $0\in
\partial\{v>\ep/\tau\}$ and $\mathcal{L}v=\beta_{\ep/\tau} (v)$ in
$B_1$, we have by Proposition \ref{cotagradsing}
$$|\nabla u^{\ep}(x_0)|=|\nabla v(\bar{x})|\leq C.$$
\item If $\delta_0=\mbox{dist} (x_0,\partial\{u^{\ep}>\ep\})\geq
\tau/8$, there holds that
\begin{enumerate}
\item $B_{\tau/8}(x_0)\subset \{u^{\ep}>\ep\},$ or \item
$B_{\tau/8}(x_0)\subset \{u^{\ep}\leq \ep\},$
\end{enumerate}
In the first case, $\mathcal{L}u^{\ep}=0$ in $B_{\tau/8}(x_0)$.
Therefore, $$|\nabla u^{\ep}(x_0)|\leq
C(N,g_0,\delta,\tau,g(1),L).$$ In the second case, we can apply
Lemma \ref{umenor2ep} and we have,
$$|\nabla u^{\ep}(x_0)|\leq
C(N,g_0,\delta,\tau,g(1), 2\|\beta\|_{\infty}).$$ The result is
proved.
\end{enumerate}
\end{proof}

\section{Passage to the limit}
\setcounter{equation}{0} Since we have that $|\nabla u^{\ep}|$ is
locally bounded by a constant independent of $\ep$, we have that
there exists a function $u \in Lip_{loc}(\Omega)$ such that, for a
subsequence $\ep_j\to 0$, $u^{\ep_j}\to u$. In this section we
will prove some properties of the function $u$.

We start with some technical results.

\begin{prop}\label{nose}
Let $\{u^{\ep}\}$ be a uniformly bounded family of nonnegative
solutions of \eqref{ecsing}. Then, for any sequence $\ep_{j}\to 0$
there exists a subsequence $\ep'_{j}\to 0$ and $u\in
\mbox{Lip}_{loc}(\Omega)$ such that,
\begin{enumerate}
\item $u^{\ep'_j}\to u$ uniformly in compact subsets of $\Omega$,
\item $\L u=0$ in $\Omega\cap\{u>0\}$ \item There exists a locally
finite measure $\mu$ such that
$\beta_{\ep'_j}(u^{\ep'_j})\rightharpoonup \mu$ as measures in
$\Omega'$, for every $\Omega'\subset\subset \Omega$,
 \item Assume $g_0\ge 1$. Then, $\nabla
u^{\ep'_j}\rightarrow \nabla u$ in $L_{loc}^{g_0+1}(\Omega)$,
\item $$\int_{\Omega} F(|\nabla u|) \nabla u \nabla
\varphi=-\int_{\Omega} \varphi \, d\mu$$ for every $\varphi\in
C^{\infty}_0(\Omega)$. Moreover $\mu$ is supported on $\Omega\cap
\partial\{u>0\}$.
\end{enumerate}
 \end{prop}

 \begin{remark} We can always assume that $g_0\ge 1$. If we don't want to assume it, we can change the statement in
  item (3) by
$\nabla u^{\ep'_j}\rightarrow \nabla u$ in
$L_{loc}^{g_1+1}(\Omega)$, where $g_1=\mbox{max\,}(1,g_0)$.
\end{remark}

\begin{proof}
(1) follows by Corollary \ref{gradsing}.

In order to prove (2), take $E\subset\subset E'\subset\subset
\{u>0\}$. Then, $u\geq c>0$ in $E'$. Therefore, $u^{\ep'_j}>c/2$
in $E'$ for $\ep_j'$ small. If we take $\ep_j'<c/2$ --as
$\mathcal{L}u^{\ep_j'}=0$ in $\{u^{\ep_j'}>\ep_j'\}$-- we have
that $\mathcal{L}u^{\ep_j'}=0$ in $E'$. Therefore, by the results
in \cite{Li1}, $\|u^{\ep_j'}\|_{C^{1,\alpha}(E)}\leq C$.

Thus, for a subsequence we have,
$$\nabla u^{\ep_j'}\rightarrow \nabla u \quad \mbox{ uniformly  in } E.$$

Therefore, $\mathcal{L}u=0$.

In order to prove (3), let us take $\Omega'\subset\subset \Omega$,
and $\varphi\in C_0^{\infty}(\Omega)$ with $\varphi=1$ in
$\Omega'$ as a test function in ($ P_{\ep_j}$). Since  $\|\nabla
u^{\ep_j'}\| \leq C$ in $\Omega'$, there holds that
$$C(\varphi)\geq \int_{\Omega} \beta_{\ep_j'}(u^{\ep_j'})\varphi\, dx
\geq \int_{\Omega'} \beta_{\ep_j'}(u^{\ep_j'})\, dx.$$ Therefore,
$ \beta_{\ep_j'}(u^{\ep_j'})$ is bounded in $L^1_{loc}(\Omega)$,
so that, there exists a locally finite  measure $\mu$ such that
$$\beta_{\ep_j'}(u^{\ep_j'})\rightharpoonup\mu \quad \mbox{ as
measures }$$that is, for every $\varphi \in C_0(\Omega)$,
$$\int_{\Omega}\beta_{\ep_j'}(u^{\ep_j'})\varphi\, dx \rightarrow \int_{\Omega} \varphi\, d\mu$$

 We divide the proof of (4) into several steps.

Let $\Omega'\subset\subset \Omega$, then by Corollary
\ref{gradsing}, $|\nabla u ^{\ep_j}|\leq C$ in $\Omega'$.
Therefore for a subsequence $\ep'_j$ we have that there exists
$\xi\in (L^{\infty}(\Omega'))^N$ such that,
\begin{equation}\label{limitelim}
\begin{aligned} &\nabla u^{\ep_j'}
\rightharpoonup \nabla u \quad  &*-\mbox{weakly in }
(L^\infty(\Omega'))^N\\
  &A(\nabla
u^{\ep_j'})\rightharpoonup\xi \quad  &*-\mbox{weakly in }
(L^\infty(\Omega'))^N\\
&u^{\ep'_j}\to u &\mbox{ uniformly in } \Omega'
\end{aligned}\end{equation}
  where $A(p)=F(|p|) p$. For simplicity we call $\ep_j'=\ep$.

%
%

Step 1.  Let us first prove that for any $v\in W^{1,G}_0(\Omega')$
there holds that
\begin{equation}\label{primero}
\int_{\Omega'}(\xi-A(\nabla u))\nabla v\,dx=0.
\end{equation}

In fact, as  $A$ is monotone (i.e $
\big(A(\eta)-A(\zeta)\big)\cdot(\eta-\zeta)\geq 0\  \forall
\eta,\zeta\in\R^N$)
 we have that, for any $w\in W^{1,G}(\Omega')$,
\begin{equation}\label{monotono1}
I=\int_{\Omega'} \big(A(\nabla u^{\ep})-A(\nabla w)\big) (\nabla
u^{\ep}-\nabla w)\, dx \geq 0.
\end{equation}
Therefore, if $\psi\in C_0^\infty(\Omega')$,
\begin{equation}\label{monotono2}\begin{aligned} &-\int_{\Omega'}
\beta_{\ep}(u^{\ep}) u^{\ep}\, dx -\int_{\Omega'} A(\nabla
u^{\ep}) \nabla w\, dx -\int_{\Omega'} A(\nabla w) (\nabla
u^{\ep}-\nabla w)\, dx \\&=-\int_{\Omega'} \beta_{\ep}(u^{\ep})
u^{\ep}\, dx -\int_{\Omega'}
A(\nabla u^{\ep}) \nabla u^{\ep}\, dx +I\\
&=-\int_{\Omega'} \beta_{\ep}(u^{\ep}) u\, dx-\int_{\Omega'}
\beta_{\ep}(u^{\ep}) (u^{\ep}-u)\psi\, dx -\int_{\Omega'}
\beta_{\ep}(u^{\ep}) (u^{\ep}-u) (1-\psi)\, dx \\&\ \ \
-\int_{\Omega'}
A(\nabla u^{\ep}) \nabla u^{\ep}\, dx+I \\
&\geq-\int_{\Omega'} \beta_{\ep}(u^{\ep}) u\, dx+\int_{\Omega'}
A(\nabla u^{\ep}) \nabla(u^{\ep}-u) \psi\, dx+\int_{\Omega'}
A(\nabla u^{\ep}) (u^{\ep}-u)\nabla \psi\, dx\\& \ \  \
-\int_{\Omega'} \beta_{\ep}(u^{\ep}) (u^{\ep}-u) (1-\psi)\, dx
 -\int_{\Omega'}
A(\nabla u^{\ep}) \nabla u^{\ep}\, dx,
\end{aligned}\end{equation}
where in the last inequality we are using \eqref{monotono1} and
\eqref{ecsin}.

Now, take  $\psi=\psi_j\to \chi_{\Omega'}$. If $\Omega'$ is smooth
we may assume that $\int |\nabla \psi_j|\, dx \to \mbox{Per
}\Omega'$. Therefore,
$$
\Big|\int_{\Omega'}A(\nabla u^{\ep}) (u^{\ep}-u)\nabla \psi_j\, dx
\Big|\leq C\|u^{\ep}-u\|_{L^{\infty}(\Omega')} \int_{\Omega'}
|\nabla \psi_j|\, dx\leq C \|u^{\ep}-u\|_{L^{\infty}(\Omega')}.
$$
So that, with this choice of $\psi=\psi_j$ in \eqref{monotono2} we
obtain,

\begin{align*}
&-\int_{\Omega'} \beta_{\ep}(u^{\ep}) u^{\ep}\, dx -\int_{\Omega'}
A(\nabla u^{\ep}) \nabla w\, dx -\int_{\Omega'} A(\nabla w)
(\nabla u^{\ep}-\nabla w)\, dx  \\&\geq -\int_{\Omega'}
\beta_{\ep}(u^{\ep}) u\, dx+\int_{\Omega'} A(\nabla u^{\ep})
\nabla(u^{\ep}-u) \, dx- C \|u^{\ep}-u\|_{L^{\infty}(\Omega')}
-\int_{\Omega'} A(\nabla u^{\ep}) \nabla u^{\ep}\, dx
\\&=-\int_{\Omega'}
\beta_{\ep}(u^{\ep}) u\, dx-\int_{\Omega'} A(\nabla u^{\ep})
\nabla u \, dx- C \|u^{\ep}-u\|_{L^{\infty}(\Omega')}.
\end{align*}

 Therefore, letting $\ep\to 0$ we get by using \eqref{limitelim} and
 (3) that,
\begin{align*} &-\int_{\Omega'} u\, d\mu
-\int_{\Omega'} \xi \nabla w\, dx -\int_{\Omega'} A(\nabla w)
(\nabla u-\nabla w)\, dx \geq -\int_{\Omega'}  u\, d\mu
-\int_{\Omega'} \xi \nabla u\, dx
\end{align*}
and then,
\begin{equation}\label{ok}
 \int_{\Omega'} (\xi-A(\nabla w)) (\nabla u-\nabla w)\,
dx \geq 0.
\end{equation}
Take now $w=u-\lambda v$ with $v\in W^{1,G}_0(\Omega')$. Dividing
by $\lambda$ and taking $\lambda\to 0^+$ in \eqref{ok} we obtain,
$$\int_{\Omega'} (\xi-A(\nabla u)) \nabla v\, dx
\geq 0.
$$
Replacing $v$ by $-v$ we obtain \eqref{primero}.

 Step 2. Let us prove that
 $\int_{\Omega'}A(\nabla\uep)\nabla \uep\to\int_{\Omega'}A(\nabla u)\nabla u$.

 \medskip

By passing to the limit in the equation
\begin{equation}\label{Auep}
0=\int_{\Omega'} A(\nabla u^{\ep}) \nabla \phi+\int_{\Omega'}
\beta_{\ep}(u^\ep)  \phi\, dx,
\end{equation}
 we have, by Step 1, that for every $\phi\in C_0^{\infty}(\Omega')$,
\begin{equation}\label{Auu}
 0=\int_{\Omega'}
A(\nabla u) \nabla \phi+\int_{\Omega'}  \phi\, d\mu.
\end{equation}

 On the other hand, taking $\phi=u^{\ep} \psi$ in \eqref{Auep} with $\psi \in C_0^{\infty}(\Omega')$ we have that
$$0=\int_{\Omega'} A(\nabla u^{\ep}) \nabla u^{\ep}
\psi\, dx + \int_{\Omega'} A(\nabla u^{\ep}) u^{\ep} \nabla\psi\,
dx+\int_{\Omega'} \beta_{\ep}(u^{\ep})  u^{\ep}\psi\, dx.$$ Using
that,
\begin{align*}
\int_{\Omega'} A(\nabla u^{\ep}) u^{\ep} \nabla\psi\, dx & \to
\int_{\Omega'}
A(\nabla u) u \nabla\psi\, dx \\
 \int_{\Omega'} \beta_{\ep}( u^{\ep}) u^{\ep}
\psi\, dx &\to \int_{\Omega'}  u \psi d\mu
\end{align*}
we obtain $$0=\lim_{\ep\to 0} \Big(\int_{\Omega'} A(\nabla
u^{\ep}) \nabla u^{\ep} \psi\, dx \Big)+ \int_{\Omega'} A(\nabla
u) u \nabla\psi\, dx+
 \int_{\Omega'} u \psi d\mu  . $$
Taking now, $\phi=u \psi$ in \eqref{Auu} we have, $$
0=\int_{\Omega'} A(\nabla u) \nabla u \psi\, dx + \int_{\Omega'}
A(\nabla u) u \nabla\psi\, dx+
 \int_{\Omega'} u \psi\, d\mu.
$$ Therefore, $$\lim_{\ep\to 0} \int_{\Omega'} A(\nabla u^{\ep})
\nabla u^{\ep} \psi\, dx= \int_{\Omega'} A(\nabla u) \nabla u
\psi\, dx. $$ Then,
\begin{align*}&\left|\int_{\Omega'}
(A(\nabla u^{\ep}) \nabla u^{\ep}-A(\nabla u) \nabla u )\,
dx\right|\\& \leq \left|\int_{\Omega'} (A(\nabla u^{\ep}) \nabla
u^{\ep}-A(\nabla u) \nabla u )\psi\, dx\right|+
\left|\int_{\Omega'} (A(\nabla u^{\ep}) \nabla u^{\ep} )(1-\psi)\,
dx\right|\\&\ \ +\left|\int_{\Omega'} A(\nabla u) \nabla u
(1-\psi)\, dx\right|
\\& \leq \left|\int_{\Omega'} (A(\nabla u^{\ep}) \nabla
u^{\ep}-A(\nabla u) \nabla u )\psi\, dx\right|+ C\int_{\Omega'}
|1-\psi|\, dx.
\end{align*}
So that, taking $\ep\to 0$ and then $\psi\to 1$ a.e with
$0\leq\psi\leq 1$ we obtain,
\begin{equation}\label{step11}
\int_{\Omega'} A(\nabla u^{\ep}) \nabla u^{\ep} dx\rightarrow
\int_{\Omega'} A(\nabla u) \nabla u\, dx.
\end{equation}

With similar ideas we can prove that,
\begin{equation}\label{step12}
\int_{\Omega'} A(\nabla u^{\ep}) \nabla u \,  dx\rightarrow
\int_{\Omega'} A(\nabla u) \nabla u\, dx.\end{equation}

Step 3.  Let us prove that
\begin{equation}\label{Guepau}\int_{\Omega'} G(|\nabla
u^{\ep}|)\, dx \rightarrow \int_{\Omega'} G(|\nabla u|)\,
dx.\end{equation}

 First, by the monotonicity of $A$  we have,
\begin{align*}\int_{\Omega'} G(|\nabla u^{\ep}|)\, dx-
\int_{\Omega'} G(|\nabla u|)\, dx&= \int_{\Omega'} \int_0^1
A(\nabla u + t(\nabla u^{\ep}-\nabla u)) \nabla(u^{\ep}-u) \, dx\\
&\geq \int_{\Omega'} A(\nabla u) \nabla (u^{\ep}-u)\,
dx.\end{align*} Therefore,  we have
$$\liminf_{\ep\to 0}\int_{\Omega'} G(|\nabla u^{\ep}|)\, dx-
\int_{\Omega'} G(|\nabla u|)\, dx\geq 0.$$

Now, by Step 2 we have,
\begin{align*}\int_{\Omega'} G(|\nabla u^{\ep}|)\, dx-
\int_{\Omega'} G(|\nabla u|)\, dx&= \int_{\Omega'} \int_0^1
A(\nabla u + t(\nabla u^{\ep}-\nabla u)) \nabla(u^{\ep}-u) \,
dx\\&\leq \int_{\Omega'} A(\nabla u^{\ep}) \nabla (u^{\ep}-u)\,
dx\rightarrow 0.
\end{align*}

Thus,   we have that \eqref{Guepau} holds.

 Step 4. End of the proof of (4).

Let $u^s=s u+(1-s)u^{\ep}$. Then,
\begin{equation}\label{cong_0}
\begin{aligned} \ \ \ \ \ \ \ &\ \ \ \  \  \int_{\Omega'} G(|\nabla u|)\, dx-
\int_{\Omega'} G(|\nabla u^{\ep}|)\, dx= \int_{\Omega'} \int_0^1
A(\nabla u^s) \nabla(u-u^{\ep}) \, ds\, dx \\  \  &\  \ =
\int_{\Omega'} \int_0^1 (A(\nabla u^s)-A(\nabla u^{\ep}))
\nabla(u^s-u^{\ep}) \, \frac{ds}s\, dx + \int_{\Omega'} A(\nabla
u^{\ep}) \nabla(u-u^{\ep}) \, dx.
\end{aligned}
\end{equation}
As in the proof of Theorem 4.1 in \cite{MW1}, we have that
\begin{align*}
\int_{\Omega'} \int_0^1 & (A(\nabla u^s)-A(\nabla u^{\ep}))
\nabla(u^s-u^{\ep}) \, ds\, dx  \\ &\geq  C\Big(\int_{ A_2}
G(|\nabla u-\nabla u^{\ep}|)  \, dx+\int_{A_1}F(|\nabla u|)
|\nabla u-\nabla u^{\ep}|^2 \, dx\Big),
\end{align*}
where
$$A_1=\{x\in \Omega': |\nabla u-\nabla u^{\ep}|\leq 2 |\nabla u|\},\quad
A_2=\{x\in \Omega': |\nabla u-\nabla u^{\ep}| > 2 |\nabla u|\}.$$

Therefore, by \eqref{step11}, \eqref{step12}, \eqref{Guepau} and
\eqref{cong_0} we have,
$$\Big(\int_{ A_2}
G(|\nabla u-\nabla u^{\ep}|)  \, dx+\int_{A_1}F(|\nabla u|)
|\nabla u-\nabla u^{\ep}|^2 \, dx\Big) \rightarrow 0.$$ Then, if
we prove that
$$\Big(\int_{ A_2}
G(|\nabla u-\nabla u^{\ep}|)  \, dx+\int_{A_1}F(|\nabla u|)
|\nabla u-\nabla u^{\ep}|^2 \, dx\Big)\geq C \int_{\Omega'}
|\nabla u-\nabla u^{\ep}|^{g_0+1}\, dx$$ the result follows.

In fact, for every $C_0>0$ there exists $C_1>0$ such that
$g(t)\geq C_1 t^{g_0}$ if $t\leq C_0$. Let $C_0$ be such that
$|\nabla u|\leq C_0$ and $|\nabla u-\nabla u^{\ep}|\leq C_0$.
Then, by Lemma \ref{prop},
\begin{align*} &G(|\nabla \uep-\nabla u|)\geq C |\nabla \uep-\nabla
u|^{g_0+1}\\&
 F(|\nabla u|)\geq C_1 |\nabla u|^{g_0-1}\geq C |\nabla \uep-\nabla
 u|^{g_0-1}\quad \mbox{ in } A_1.
\end{align*}
and the claim follows.

Finally (5) holds by (4), (3) and  (2).
 \end{proof}
\begin{lema}\label{bolwepj}
Let $\{u^{\ep_{j}}\}$ be a uniformly bounded family of solutions
of $(P_{\ep_j})$ in $\Omega$ such that $u^{\ep_j}\to u$ uniformly
on compact subsets of $\Omega$ and $\ep_j\to 0$. Let $x_0,\ x_n\in
\Omega\cap\partial\{u>0\}$ be such that $x_n\to x_0\ \mbox{ as }
n\to \infty$. Let $\lambda_n\to 0$,
$u_{\lambda_n}(x)=\frac{1}{\lambda_n} u(x_n+\lambda_n x)$ and
 $(u^{\ep_j})_{\lambda_n}(x)=\frac{1}{\lambda_n} u^{\ep_j}(x_n+\lambda_n x)$.
 Suppose that $u_{\lambda_n}\to U$ as $n\to \infty$ uniformly on compact sets of $\R^N$.
 Then, there  exists $j(n)\to \infty$ such that for every $j_n\geq j(n)$ there holds that $\ep_{j_n}/\lambda_n\to 0$ and
\begin{enumerate}
\item $(u^{\ep_{j_n}})_{\lambda_n}\to U$ uniformly in compact
subsets of $\R^N.$ \item $\nabla
(u^{\ep_{j_n}})_{\lambda_n}\rightarrow \nabla U$ in
$L_{loc}^{g_0+1}(\R^N)$,

\item $\nabla u_{\lambda_n}\rightarrow \nabla U$ in
$L_{loc}^{g_0+1}(\R^N)$.
\end{enumerate}

\end{lema}
\begin{proof} { The proof follows from Proposition \ref{nose} as the proof of Lemma 3.2 follows from Lemma 3.1 in \cite{CLW1}.}
\end{proof}

Now we prove a technical lemma that is the basis of our main
results.

\begin{lema}\label{sale}
Let $u^{\ep}$ be solutions to
$$ \L u^{\ep}=\beta_{\ep}(u^{\ep})$$ in $\Omega$. Then, for any
$\psi \in C_0^{\infty}(\Omega)$ we have,
\begin{equation}\label{salsale}
-\int_{\Omega} G(|\nabla u^{\ep}|) \psi_{x_1}\, dx + \int_{\Omega}
F(|\nabla u^{\ep}|) \nabla u^{\ep} \nabla \psi\,u^{\ep}_{x_1}\,
dx=\int_{\Omega} B_{\ep}(u^{\ep}) \psi_{x_1},\end{equation} where
$B_{\ep}(s)=\int_0^s \beta_{\ep} (\tau)\, d\tau$.
\end{lema}
\begin{proof}
For simplicity, since $\ep$ will be fixed throughout the proof, we
will denote $u^{\ep}=u$.

We know that $|\nabla u|\leq C$, for some constant $C$.
 Take $g_n(t)=g(t)+\frac{t}{n}$, %
then
\begin{equation}\label{mismaclase}\mbox{min}\{1,\delta\}\leq\frac{{g'_n}(t)
t}{g_n(t)}\leq \mbox{max}\{1,g_0\}.\end{equation} Take
$A_n(p)=\frac{g_n(|p|)}{|p|}p$, and
$\mathcal{L}_n(v)=\mbox{div}(A_n(\nabla v))$. For
$\Omega'\subset\subset\Omega$ let us take $u_n$ the solution of
\begin{equation}\label{ecunn}\begin{cases} \mathcal{L}_n u_n=\beta_{\ep}(u) &\quad
\mbox{ in } \Omega'
\\ u_n=u &\quad \mbox{ on } \partial\Omega'.
\end{cases}
\end{equation}

By  \eqref{mismaclase},we have  that  all the $g_n's$ belong to
the same class and then, by the results of \cite{Li1} we have that
for every $\Omega''\subset\subset\Omega'$  there exists a constant
$C$ independent of $n$ such that
$\|u_n\|_{C^{1,\alpha}(\Omega'')}\leq C$.

Therefore, there exists $u_0$ such that, for a subsequence
\begin{align*}&u_n\rightarrow u_0 \quad \mbox{uniformly on compact subsets of } \Omega'\\
&\nabla u_n\rightarrow \nabla u_0 \quad \mbox{uniformly on compact
subsets of }\Omega'.
\end{align*}
On the other hand,  $A_n(p)\rightarrow {A}(p)$  uniformly in
compact sets of $\R^N$.
 Thus, $\mathcal{{L}} u_0= \beta_{\ep}(u)$ and, as $u_0=u$ on
$\partial\Omega'$ in the sense of $W^{1,G}(\Omega')$ and
$\mathcal{{L}} u= \beta_{\ep}(u)$, there holds that  $u_0=u$ in
$\Omega'$. (Observe that in the proof of the Comparison Principle,
in Lemma 2.8 of \cite{MW1} we can change the equation $\L u=0$, by
$\L u=f(x)$ with $f\in L^{\infty}(\Omega)$ to prove uniqueness of
solution of the Dirichlet problem).

Now let us prove that the following equality holds,
$$-\int_{\Omega} G_n(|\nabla u_n|) \psi_{x_1}\, dx + \int_{\Omega} F_n(|\nabla u_n|)
\nabla u_n \nabla \psi\, {u_n}_{x_1}\, dx=-\int_{\Omega}
\beta_{\ep}(u) {u_n}_{x_1} \psi.$$

In fact, for  $n$ fixed we have that $F_n(t)=g_n(t)/t\geq 1/n$ and
then by the uniform estimates of \cite{LU},  $u_n\in
W^{2,2}(\Omega)$. As $u_n$ is a weak solution of \eqref{ecunn} and
as $u_n\in W^{2,2}(\Omega)$, taking as test function in the weak
formulation of \eqref{ecunn} the function $\psi {u_n}_{x_1}$, we
have that
$$\int_{\Omega} F_n(|\nabla u_n|)\nabla u_n\nabla(  \psi {u_n}_{x_1})\, dx =-\int_{\Omega}
\beta_{\ep}(u){u_n}_{x_1} \psi\, dx.$$

As $(G_n(|\nabla u_n|))_{x_1}=g_n(|\nabla u_n|)\frac{\nabla
u_n}{|\nabla u_n|} ({\nabla u_n})_{x_1}=F(|\nabla u_n|) \nabla u_n
({\nabla u_n})_{x_1}$ we have that
$$-\int_{\Omega} G_n(|\nabla u_n|) \psi_{x_1}\, dx+
 \int_{\Omega} F_n(|\nabla u_n|)\nabla u_n\nabla \psi\,
{u_n}_{x_1} \, dx =-\int_{\Omega} \beta_{\ep}(u) u_{nx_1}\psi\,
dx,$$

Passing to the limit as $n\to \infty$ and then, integrating by
parts on the right hand side we get,
$$-\int_{\Omega} G(|\nabla u|) \psi_{x_1}\, dx+
 \int_{\Omega} F(|\nabla u|)\nabla u\nabla \psi\,
{u}_{x_1} \, dx =\int_{\Omega} B_{\ep}(u) \psi_{x_1}\, dx.
$$

\end{proof}

Now, we characterize some special global limits.

\begin{prop}\label{alphamenorphiM}
Let $x_0\in \Omega$ and let $u^{\ep_k}$ be solutions to
$$ \L u^{\ep_k}=\beta_{\ep_k}(u^{\ep_k})$$ in $\Omega$. If $u^{\ep_k}$
converge to $\alpha(x-x_0)_1^+$ uniformly in compact subsets of
$\Omega$, with
  $\ep_k\to 0$ as $k\to \infty$ and $\alpha\in \R$, there holds that
$$\alpha=0 \quad\mbox{or} \quad \alpha=\Phi^{-1}(M).$$
Where $\Phi(t)=g(t)t-G(t)$.
\end{prop}

\begin{proof}
Assume, for simplicity, that $x_0=0$. Since $u^{\ep_k}\geq 0$, we
have that $\alpha\geq 0$. If $\alpha=0$ there is nothing to prove.
So let us assume that $\alpha>0$. Let $\psi \in
C_0^{\infty}(\Omega)$. By Lemma \ref{sale} we have,
\begin{equation}\label{novale}
-\int_{\Omega} G(|\nabla u^{\ep_k}|) \psi_{x_1}\, dx +
\int_{\Omega} F(|\nabla u^{\ep_k}|) \nabla u^{\ep_k} \nabla \psi
\, u^{\ep_k}_{x_1}\, dx=\int_{\Omega} B_{\ep_k}(u^{\ep_k})
\psi_{x_1}.
\end{equation}
Since $0\leq B_{\ep_k}(s)\leq M$, there exists $M(x)\in
L^{\infty}(\Omega)$, $0\leq M(x)\leq M$, such that
$B_{\ep_k}\rightarrow M$ $*\,$-- weakly in $L^{\infty}(\Omega)$.

If $y\in \Omega\cap\{x_1>0\}$, then $u^{\ep_k}\geq \frac{\alpha
y_1}{2}$ in a neighborhood of $y$ for $k$ large. Thus,
$u^{\ep_k}\geq \ep_k$ and  we have
$$B_{\ep_k}(u^{\ep_k})(x)=\int_0^{u^{\ep_k}/{\ep_k}}\beta(s)\,
ds=M.$$

On the other hand, if we let $K\subset\subset\Omega\cap\{x_1<0\}$,
since by Proposition \ref{nose} $\beta_{\ep_k}(u^{\ep_k})\to 0$ in
$L^1(K)$, we have that $\int_K\big|\nabla
B_{\ep_k}(u^{\ep_k})\big|\,dx=\int_K\beta_{\ep_k}(u^{\ep_k})|\nabla
u^{\ep_k}|\,dx\to0$. Therefore,
 we may assume that $B_{\ep_k}\rightarrow \overline M$ in $L^1_{loc}(\{x_1<0\})$ for a constant $\overline M\in [0,M]$.

Passing to the limit in \eqref{novale}, using the strong
convergence result in Proposition \ref{nose} we have
$$
-\int_{\{x_1>0\}} G(\alpha) \psi_{x_1}\, dx + \int_{\{x_1>0\}}
F(\alpha)\, \alpha^2 \psi_{x_1} \, dx=M\int_{\{x_1>0\}}
 \psi_{x_1}+\overline{M}\int_{\{x_1<0\}}
 \psi_{x_1}.
$$
Then,
$$
 (-G(\alpha)  + g(\alpha) \alpha)\int_{\{x_1>0\}}
  \psi_{x_1} \, dx=M\int_{\{x_1>0\}}
 \psi_{x_1}\, dx+\overline{M}\int_{\{x_1<0\}}
 \psi_{x_1}\, dx.
$$
And, integrating by parts, we obtain
$$
 (-G(\alpha)  + g(\alpha) \alpha)\int_{\{x_1=0\}}
 \psi \, dx'=M\int_{\{x_1=0\}}
 \psi \, dx'-\overline{M}\int_{\{x_1=0\}}
 \psi\, dx'.
$$

Thus, $
 (-G(\alpha)  + g(\alpha) \alpha)=M-\overline{M}$.

 In order to see that
 $\alpha=\Phi^{-1}(M)$ let us show that
$\overline{M}=0$.

In fact, let $K\subset\subset \{x_1<0\}\cap \Omega$. Then for any
$\eta>0$ there exists $0<\delta<1$ such that,
\begin{align*}
\big|K\cap\{\eta<B_{\ep_j}(u^{\ep_j})<M-\eta\}\big|&\leq
\big|K\cap\{\delta<u^{\ep_j}/{\ep_j}<1-\delta\}\big|\leq
\big|K\cap\{\beta_{\ep_j}(u^{\ep_j})\geq a/\ep_j\}\big|\to 0
\end{align*}
 as $j\to \infty$,
where $a=\inf_{[\delta,1-\delta]}\beta>0$, and we are using that
$\beta_{\ep_j}(u^{\ep_j})$ is bounded in $L^1(K)$ uniformly in
$j$.

Now, as $B(u^{\ep_j})\to \overline{M}$ in $L^1(K)$, we conclude
that
$$\big|K\cap\{\eta<\overline{M}<M-\eta\}\big|=0$$ for every $\eta>0$. Hence, $\overline{M}=0$ or $\overline{M}=M$
and, since $\alpha>0$, we must have  $\overline{M}=0$.
\end{proof}

\bigskip

\begin{prop}\label{gammapositiva}
Let $x_0\in\Omega$, and let $u^{\ep_k}$ be a solution to
$\mathcal{L}u^{\ep_k}=\beta_{\ep_k}(u^{\ep_k})$ in $\Omega$.
Assume $g'$ satisfies \eqref{condderivadag} below. If $u^{\ep_k}$
converges to $\alpha(x-x_0)_1^++\gamma(x-x_0)_1^-$ uniformly in
compact subsets of $\Omega$, with $\alpha,\gamma>0$ and $\ep_k\to
0$ as $k\to\infty$, then
$$\alpha=\gamma\leq \Phi^{-1}(M).$$
\end{prop}
\begin{proof}
We can assume that $x_0=0$.

As in the proof of Proposition \ref{alphamenorphiM} we see that
$B_{\ep_k}(\uepk)\to M$ uniformly on compact sets of $\{x_1>0\}$
and $\{x_1<0\}$. Since $\uepk$ satisfies \eqref{salsale} we get,
after passing to the limit, for any $\psi\in C_0^\infty(\Omega)$,
$$-\int_{\{x_1>0\}} \Phi(\alpha) \psi_{x_1}\, dx  -\int_{\{x_1<0\}} \Phi(\gamma) \psi_{x_1}\, dx
=\int_{\Omega} M \psi_{x_1}.$$

Integrating by parts we obtain,
$$\int_{\{x_1=0\}} \Phi(\alpha) \psi\, dx'  -\int_{\{x_1=0\}} \Phi(\gamma) \psi\,
dx'=0$$ and then, $\alpha=\gamma$.

\bigskip

Now assume that $\alpha>\Phi^{-1}(M)$. We will prove that this is
a contradiction.

{\bf Step 1}. Let $\mathcal{R}_2=\{x=(x_1,x')\in \mathbb{R}^N:
|x_1|<2, |x'|<2\}.$ From the scaling invariance of the problem, we
can assume that $\mathcal{R}_2\subset \Omega$.

We will construct a family  $\{v^{\ep_j}\}$ of solutions of
$(P_{\ep_j})$ in $\mathcal{R}_2$ satisfying
$v^{\ep_j}(x_1,x')=v^{\ep_j}(-x_1,x')$ in $\mathcal{R}_2$, and
such that $v^{\ep_j}\to u$ uniformly on compact subsets of
$\mathcal{R}_2$, where $u(x)=\alpha |x_1|$.

To this end, we take $b_{\ep_j}=\sup_{\mathcal{R}_2}|u^{\ep_j}-u|$
and $v^{\ep_j}$ the minimal solution (the minimum of all
supersolutions) to $(P_{\ep_j})$ in $\mathcal{R}_2$ with boundary
values $v^{\ep_j}=u-b_{\ep_j}$ on $\partial\mathcal{R}_2.$

By Proposition \ref{nose}, there exists $v\in
Lip_{loc}(\mathcal{R}_2)$ such that, for a subsequence that we
still denote $v^{\ep_j}$, $v^{\ep_j}\to v$ uniformly on compact
subsets of $\mathcal{R}_2$. From the minimality of $v^{\ep_j}$ we
have that $u\geq v$.

{ In order to prove that $u\leq v$, we  considered two cases.

First suppose that
$\alpha>\Phi^{-1}\big(\frac{g_0}{\delta}M\big)$}.  Let $w\in
C^{1,\beta}(\mathbb{R})$, be the solution to {
$$(F(|w'|) w')'=\frac{g_0}{\delta}\beta(w) \quad \mbox{in } \mathbb{R}, \quad w(0)=1, \quad
w'(0)=\alpha.$$}

Observe that, when $w'(s)> 0$, the equation is locally uniformly
elliptic so that, as long as $w'>0$, there holds that $w\in C^2$
and a solution to
$$
\big(g(w')\big)'=\frac{g_0}{\delta}\beta(w).
$$

\medskip

Suppose that there exists an $s\in \mathbb{R}$ such that
$w'(s)=0$. Take $s_1$ as the supremum of the $s$'s such that this
happens. Then, $s_1<0$ and,
in $(s_1,0]$, $w'>0$ and $F(|w'|)\,w'=g(w')$. Multiplying the equation by $w'$ and integrating  in this interval we get, 
$$
-\int_{s_1}^0 g(w') w''+g(w')w'\Big|_{s_1}^0 =
{\frac{g_0}{\delta}}B(w)\Big|_{s_1}^0.
$$

Since $g(w')w''=\big (G(w')\big)'$ we get,
$$
\Phi(\alpha)={\frac{g_0}{\delta}}M-{\frac{g_0}{\delta}}B\big(w(s_1)\big)\le
{\frac{g_0}{\delta}}M
$$
which  is a contradiction.

 Then, $w'>0$
everywhere. By the same calculation as before, we obtain that for
any $s\in \mathbb{R}$ we have,
$$\Phi(w'(s))=\Phi(\alpha)+ {\frac{g_0}{\delta}}B(w(s))-{\frac{g_0}{\delta}}M\leq \Phi(\alpha),$$
 and
\begin{equation}\label{wsalpha}
 \Phi(w'(s))=\Phi(\alpha)+ {\frac{g_0}{\delta}}B(w(s))-{\frac{g_0}{\delta}}M\geq
 \Phi(\alpha)-{\frac{g_0}{\delta}}M=\Phi(\bar{\alpha}),
\end{equation}
  for some $\alpha>\bar{\alpha}>0$.
 Thus, $\bar{\alpha}\leq w'(s)\leq \alpha$.

Therefore, $w'(s)=\alpha$ for $s\geq 0$ and there exists
$\bar{s}<0$ such that $w(\bar{s})=0$. This implies, by
\eqref{wsalpha},
 that $w'(\bar{s})=\bar{\alpha}$, and then
$w'(s)=\bar{\alpha}$ for all $s\leq \bar{s}$. Therefore,
$$w(s)=\begin{cases}
1+\alpha s &\quad s>0 \\
\bar{\alpha} (s-\bar{s}) &\quad s\leq \bar{s}.
\end{cases}$$
Let $w^{\ep_j}(x_1)=\ep_j w\Big(\frac{x_1}{\ep_j}
-\frac{b_{\ep_j}}{\bar{\alpha} \ep_j}+\bar{s}\Big)$ then,
$$w^{\ep_j}(0)=\ep_j w\Big(-\frac{b_{\ep_j}}{\bar{\alpha}
\ep_j}+\bar{s}\Big)=\ep_j \bar{\alpha}
\Big(\bar{s}-\frac{b_{\ep_j}}{\bar{\alpha}
\ep_j}-\bar{s}\Big)=-b_{\ep_j}$$ and ${w^{\ep_j}}'(s)\leq \alpha$.
Therefore, $w^{\ep_j}\leq u-b_{\ep_j}$ in $\mathbb{R}$ so that,
$w^{\ep_j}\leq v^{\ep_j}$ on $\partial \mathcal{R}_2$.

Then, by the comparison principle below (Lemma \ref{compara}), we
have that $w^{\ep_j}\leq v^{\ep_j}$ in $\mathcal{R}_2$.

Take $x_1>0$. Then, for $j$ large
$x_1-\frac{b_{\ep_j}}{\bar{\alpha}}>\frac{x_1}{2}$. Thus,
$\frac{1}{\ep_j}
(x_1-\frac{b_{\ep_j}}{\bar{\alpha}})+\bar{s}>\frac{x_1}{2\ep_j}+\bar{s}>0$
for $j$ large.

Therefore, $w^{\ep_j}(x)=\ep_j+\alpha x_1-\frac{\alpha
}{\bar{\alpha}}b_{\ep_j}+\alpha\ep_j\bar{s}$. Hence,
$w^{\ep_j}\rightarrow u$ uniformly on compact set of $\{x_1>0\}$.

Passing to the limit, we get
  that $u\leq v$ in $\mathcal{R}_2 \cap \{x_1>0\}$. Observe that,
  by the uniqueness of the minimal solution, we have
that $v^{\ep_j}(x_1,x')=v^{\ep_j}(-x_1,x')$. Thus, we obtain that
$u\leq v$ in $\mathcal{R}_2$.

This completes the first case.

\medskip

{ Now, suppose that $\alpha\leq
\Phi^{-1}\big(\frac{g_0}{\delta}M\big)$. Let $w\in
C^{1,\beta}(\mathbb{R})$, satisfying
$$(F(|w'|) w')'=\beta(w) \quad \mbox{in } \mathbb{R}, \quad w(0)=1, \quad
w'(0)=\alpha.$$ Again, when $w'(s)> 0$ the equation is locally
uniformly elliptic and then $w\in C^2$.


Proceeding as in the first case we see that $\bar \alpha\le
w'(s)<\alpha$ in $\R$ where, in the present case,
$\Phi(\bar\alpha)=\Phi(\alpha)-M$.

%
%

In this way we see that there exists $\bar s<0$ such that
$$w(s)=\begin{cases}
1+\alpha s &\quad s>0 \\
\bar{\alpha} (s-\bar{s}) &\quad s\leq \bar{s}.
\end{cases}$$
Let $w^{\ep_j}(x_1)=\ep_j w\big(\frac{x_1}{\ep_j}
-\frac{b_{\ep_j}}{\bar{\alpha} \ep_j}+\bar{s}\big)$, then
$$w^{\ep_j}(0)=\ep_j w\big(-\frac{b_{\ep_j}}{\bar{\alpha}
\ep_j}+\bar{s}\big)=\ep_j \bar{\alpha}
\big(\bar{s}-\frac{b_{\ep_j}}{\bar{\alpha}
\ep_j}-\bar{s}\big)=-b_{\ep_j}$$ and ${w^{\ep_j}}'(s)\leq \alpha$.
Therefore, $w^{\ep_j}\leq u-b_{\ep_j}$ in $\mathbb{R}$, so that,
$w^{\ep_j}\leq v^{\ep_j}$ on $\partial \mathcal{R}_2$ and since
${w^{\ep_j}}'\leq \alpha\leq \Phi^{-1}\big( \frac{g_0}{\delta}M
\big)$ we have, by the comparison principle below (Lemma
\ref{compara}),  that $w^{\ep_j}\leq v^{\ep_j}$ in
$\mathcal{R}_2$. We can conclude as in the previous case that,
$u\leq v$ in $\mathcal{R}_2$. }

\bigskip

{\bf Step 2}. Let $\mathcal{R}^+=\{x: 0<x_1<1, |x'|<1\}.$ Define,
$$F_j=\int_{\partial \mathcal{R}^+\cap\{x_1=1\}} F(|\nabla
v^{\ep_j}|) (v_{x_1}^{\ep_j})^2\, dx' + \int_{\partial
\mathcal{R}^+\cap\{|x'|=1\}} F(|\nabla v^{\ep_j}|) v_{n}^{\ep_j}
v_{x_1}^{\ep_j}\, dS,$$ where ${v_n}^{\ep_j}$ is the exterior
normal of $v^{\ep_j}$ on $\partial\mathcal{R}^+\cap\{|x'|=1\}$. We
first want to prove that,
$$F_j\leq \int_{\partial\mathcal{R}^+\cap\{x_1=1\}} \Big(G(
|\nabla v^{\ep_j}|) +B_{\ep_j}(v^{\ep_j})\Big)\, dx'.
$$

In order to prove it, we proceed as in the proof of Lemma
\ref{sale}. This is, we can suppose that $F(s)\geq c>0$, by using
an approximation argument. Therefore, we can suppose that
$v^{\ep_j}\in W^{2,2}(\mathcal{R}_2)$. Multiplying equation
$(P_{\ep_j})$ by $v^{\ep_j}_{x_1}$ in $\mathcal{R}^+$ and using
the definitions of $G$ and $F$ we have,

\begin{align*}E_j&:=\int\int_{\mathcal{R}^+} \frac{\partial}{\partial
x_1}\Big(G( |\nabla v^{\ep_j}|)\Big) \, dx
=\int\int_{\mathcal{R}^+} F(|\nabla v^{\ep_j}|) \nabla v^{\ep_j}
\nabla v^{\ep_j}_{x_1}\, dx\\&=\int\int_{\mathcal{R}^+}
\mbox{div}(F(|\nabla v^{\ep_j}|) \nabla v^{\ep_j}
v_{x_1}^{\ep_j})\, dx -\int\int_{\mathcal{R}^+}
\beta_{\ep_j}(v^{\ep_j}) v^{\ep_j}_{x_1}=:H_j-G_j.\end{align*}
Using the divergence theorem and the fact that
$v^{\ep_j}_{x_1}(0,x')=0$ (by the symmetry in the $x_1$ variable)
we find that, $H_j=F_j.$

From the convergence of $v^{\ep_j}\rightarrow u=\alpha |x_1|$ in
$\mathcal{R}_2$ and Proposition  \ref{nose} we have that
$$\nabla v^{\ep_j}\rightarrow \alpha e_1 \quad \mbox{ a.e in
} \mathcal{R}^+_2=\mathcal{R}_2\cap \{x_1>0\}.$$ Since $|\nabla
v^{\ep_j}|$ are uniformly bounded, from the dominate convergence
theorem we deduce that,
\begin{equation}\label{alphaM1}
\lim_{j\to\infty}
F_j=\int_{\partial\mathcal{R}^+\cap\{x_1=1\}}g(\alpha)\alpha\, dx'
\end{equation}
and
\begin{align*} F_j&=E_j+G_j\\ &=
\int\int_{\mathcal{R}^+} \frac{\partial}{\partial x_1}\Big(G(
|\nabla v^{\ep_j}|)+B_{\ep_j}(v^{\ep_j})\Big) \, dx\\ &=
\int_{\partial\mathcal{R}^+\cap\{x_1=0\}} -\Big(G( |\nabla
v^{\ep_j}|)+B_{\ep_j}(v^{\ep_j})\Big)\, dx'+
\int_{\partial\mathcal{R}^+\cap\{x_1=1\}} \Big(G( |\nabla
v^{\ep_j}|)+B_{\ep_j}(v^{\ep_j})\Big)\, dx'\\& \leq
\int_{\partial\mathcal{R}^+\cap\{x_1=1\}} \Big(G( |\nabla
v^{\ep_j}|)+B_{\ep_j}(v^{\ep_j})\Big) \, dx'.
\end{align*}
Using again that $v^{\ep_j}\to u=\alpha |x_1|$ uniformly on
compact subsets of $\mathcal{R}_2$, we have that $|\nabla
v^{\ep_j}|\to \alpha$ uniformly  on $\partial
\mathcal{R}^+\cap\{x_1=1\}$ and $B_{\ep_j}(v^{\ep_j})=M$ on this
set for $j$ large. Therefore,
\begin{equation}\label{alphaM2}
\limsup_{j\to\infty} F_j\leq
\int_{\partial\mathcal{R}^+\cap\{x_1=1\}} (G(\alpha)+M)\
dx'.\end{equation}

Thus, from \eqref{alphaM1} and \eqref{alphaM2} we obtain
$\Phi(\alpha)\leq M$ which is a contradiction.

\end{proof}

Now, we prove the comparison principle needed in the proof of the
lemma above. This is the step  where we  need an additional
hypothesis: {There exist $\eta_0>0$ such that,
\begin{equation}\label{condderivadag}
g'(t)\le s^2 g'(ts) \quad \mbox{if}\quad 1\leq s\leq 1+\eta_0
\quad \mbox{and}\quad 0<t\le
\Phi^{-1}\Big(\frac{g_0}{\delta}M\Big).
\end{equation}}

\begin{remark}
We remark that condition \eqref{condderivadag} holds for all the
examples of functions $g$ satisfying condition \eqref{cond}
considered in the Introduction.

This is immediate when $g$ is a positive power or the sum of
positive powers.

If $g(t)=t^{a}\mbox{log\,}(b+ct)$, we have for $s\ge1$,
$$
s^2g'(ts)=
s^{a+1}at^{a-1}\mbox{log\,}(b+cts)+s^{a+2}\frac{ct^a}{b+cts}\ge
\Big[at^{a-1}\mbox{log\,}(b+ct) +\frac{s ct^a}{b+cts}\Big].
$$
Since
$$
g'(t)=at^{a-1}\mbox{log\,}(b+ct)+\frac{ct^a}{b+ct},
$$
condition \eqref{condderivadag} holds if
$$
\frac s{b+cts}\ge\frac 1{b+ct}.
$$
Or, equivalently
$$
sb+cst\ge b+cst,
$$
and this last inequality holds for $s\ge1$.

Finally, if $g\in C^1(\R)$,  $g(t)=c_1t^{a_1}$ for $t\le k$,
$g(t)=c_2t^{a_2}+c_3$ for $t>k$ we have
$$
s^2g'(ts)=\begin{cases}
s^{a_1+1}a_1c_1t^{a_1-1}\quad&\mbox{if }st\le k\\
s^{a_2+1}a_2c_2t^{a_2-1}\quad&\mbox{if }st\ge k.
\end{cases}
$$
So that,
\begin{enumerate}
\item If $t\ge k$, then $ts\ge k$ and
$$
s^2g'(ts)=s^{a_2+1}a_2c_2t^{a_2-1}\ge a_2c_2t^{a_2-1}=g'(t).
$$

\item If $ts\le k$ (i. e. $t\le  k/s$), we have, in particular,
that $t\le k$ and
$$
s^2g'(ts)=s^{a_1+1}a_1c_1t^{a_1-1}\ge a_1c_1t^{a_1-1}=g'(t).
$$

\item If $ k/s<t<k$ there holds that
$s^2g'(ts)=s^{a_2+1}a_2c_2t^{a_2-1}$ and $g'(t)=a_1c_1t^{a_1-1}$.
Therefore, condition \eqref{condderivadag} is equivalent to
\begin{equation}\label{cond1}
s^{a_2+1}\ge \frac{a_1c_1}{a_2c_2}t^{a_1-a_2}.
\end{equation}
Observe that the condition that $g'$ be continuous implies that
$\frac{a_1c_1}{a_2c_2}=k^{a_2-a_1}$. Thus, \eqref{cond1} is
equivalent to
\begin{equation}\label{cond2}
s^{a_2+1}\ge \Big(\frac tk\Big)^{a_1-a_2}.
\end{equation}

We consider two cases.
\begin{enumerate}
\item If $a_1\ge a_2$, \eqref{cond2} holds since $t<k$ and
$s\ge1$.

\item If $a_1<a_2$, as $t>k/s$ there holds that,
$$
\Big(\frac tk\Big)^{a_1-a_2}<\frac 1{s^{a_1-a_2}}\le s^{a_2+1},
$$
because $\frac 1{s^{a_1}}\le s$ since $s\ge1$.
\end{enumerate}

\end{enumerate}
\end{remark}


Let us now prove the comparison lemma used in the proof of
Proposition \ref{gammapositiva}.

\begin{lema}\label{compara}
{Let $w^{\ep}(x_1)$ in $C^2(\mathbb{R})$ such that
${w^{\ep}}'(x_1)\geq \bar{\alpha}>0$ and $v^{\ep}(x)\geq 0$ a
solution of $\mathcal{L}v^{\ep}=\beta_{\ep}(v^{\ep})$ in
$\mathcal{R}=\{x=(x_1,x'): a<x_1<b, |x'|<r\}$, continuous up to
$\partial \mathcal{R}$. Then,  the following comparison principle
holds: if $v^{\ep}(x)\geq w^{\ep}(x_1)$ for all $x\in
\partial \mathcal{R}$ and if,

\begin{enumerate}

\item
 $\mathcal{L}(w^{\ep})\geq \frac{g_0}{\delta}\beta_{\ep}(w^{\ep})$ on $\mathbb{R}$,

\medskip
or
\medskip

\item  $\mathcal{L}w^{\ep}\ge \beta_{\ep}(w^{\ep})$,
${w^{\ep}}'\leq \Phi^{-1}(\frac{g_0}{\delta}M)$ and $g'$ satisfies
condition \eqref{condderivadag},

\end{enumerate}
then, $v^{\ep}(x)\geq w^{\ep}(x_1)$ for all $x\in \mathcal{R}$. }

\end{lema}
\begin{proof}

{Since ${w^{\ep}}'(x_1)\geq \bar{\alpha}$ there exist $x_0$ such
that $w^{\ep}(x_0)=0$. Let us suppose that $x_0=0$.}
%
Since $v^{\ep}(x)\geq 0$, we can find $\tau$ such that,
$$w^{\ep}(x_1-\tau)<v^{\ep}(x) \quad \mbox{ on }
\bar{\mathcal{R}}.$$ For $\eta>0$ sufficiently small define,
$$w^{\ep,\eta}(x_1):=w^{\ep}(\varphi_{\eta}(x_1-c_{\eta})),$$ where
$\varphi_{\eta}(s)=s+\eta s^2$ and $c_{\eta}>0$ is the smallest
constant such that $\varphi_{\eta}(s-c_{\eta})\leq s$ on
$[-2\tau,2\tau]$ (observe that $c_{\eta}\to 0$ when $\eta\to 0$).
If $c_\eta-\frac{1}{\eta}\leq -2\tau$ then
$\varphi_\eta(s-c_{\eta})\leq 0$ for $s\leq c_\eta$. Observe that,
in $[-2\tau,2\tau]$, $w^{\ep,\eta}\leq w^{\ep}$ and, as $\eta\to
0$, $w^{\ep,\eta}\rightarrow w^{\ep}$ uniformly.

If we call $\widetilde{\varphi}_{\eta}(s)=
\varphi_{\eta}(s-c_{\eta})$,  we have,
\begin{equation}\label{formulafin}
\mathcal{L}w^{\ep,\eta}= g'({w^{\ep}}'(\widetilde{\varphi}_{\eta})
\widetilde{\varphi}'_{\eta})
{w^{\ep}}''(\widetilde{\varphi}_{\eta})
(\widetilde{\varphi}'_{\eta})^2+
g'({w^{\ep}}'(\widetilde{\varphi}_{\eta})
\widetilde{\varphi}'_{\eta})
{w^{\ep}}'(\widetilde{\varphi}_{\eta})
\widetilde{\varphi}_{\eta}''.
\end{equation}

{In the first case we use that, by condition \eqref{cond},  we
have for $s\geq 1$,
$$g'(ts)\geq \delta \frac{g(ts)}{ts} \geq \delta
\frac{g(t)}{ts}\geq \frac{\delta g'(t)}{g_0 s}.$$ Therefore
\begin{equation}\label{sgprima}
s^2 g'(ts)\geq \frac{\delta} {g_0} sg'(t).\end{equation}

Taking $s={{\widetilde{\varphi}}_{\eta}}'$ and
$t={w^{\ep}}'({{\widetilde{\varphi}}_{\eta}}')$, using
\eqref{formulafin}, the fact that ${{\varphi}_{\eta}}''>0$,
${w^{\ep}}'>0$ and \eqref{sgprima} we have,
$$\mathcal{L}w^{\ep,\eta}> \frac{\delta}{g_0} g'({w^{\ep}}'
(\widetilde{\varphi}_{\eta})){w^{\ep}}''(\widetilde{\varphi}_{\eta})
{{\widetilde{\varphi}}_{\eta}}'=\frac{\delta}{g_0}
\mathcal{L}w^{\ep}(\widetilde{\varphi}_{\eta}){{\widetilde{\varphi}}_{\eta}}'
\geq \beta_{\ep}(w^{\ep,\eta}){{\widetilde{\varphi}}_{\eta}}'.$$
Since, $\beta_{\ep}({w}^{\ep,\eta})=0$ when $x_1\leq c_{\eta}$ and
$\widetilde{\varphi}'_{\eta}\geq 1$ when $x_1\geq c_{\eta}$, we
have that $\mathcal{L}w^{\ep,\eta}>\beta_{\ep}({w}^{\ep,\eta})$.

\medskip

For the second case, choose $\eta$ small enough so that
$0<c_\eta\le 1$ and ${\widetilde{\varphi}_{\eta}}'(r)\leq
1+\eta_0$ for $a<r<b$.

If $x_1<c_\eta$, we proceed as in the previous case and deduce
that $\L(w^{\ep,\eta})>0=\beta_\ep({w}^{\ep,\eta})$.

If $x_1\ge c_\eta$, we can apply
 condition \eqref{condderivadag} with $s={{\widetilde{\varphi}}_{\eta}}'$ and
$t={w^{\ep}}'({{\widetilde{\varphi}}_{\eta}}')$ since
${w^{\ep}}'\leq \Phi^{-1}(\frac{g_0}{\delta}M)$.

 Then,  using that
 ${{\varphi}_{\eta}}''>0$,
${w^{\ep}}'>0$ and \eqref{formulafin} we have,
\begin{align*}
\mathcal{L}w^{\ep,\eta}> g'({w^{\ep}}'
(\widetilde{\varphi}_{\eta}))
{w^{\ep}}''(\widetilde{\varphi}_{\eta})
=\mathcal{L}w^{\ep}(\widetilde{\varphi}_{\eta})\ge
\beta_{\ep}({w}^{\ep,\eta}) .
\end{align*}
}

Summarizing, in both cases we have,
\begin{align*}
\mathcal{L}w^{\ep,\eta}>\beta_{\ep}(w^{\ep,\eta}),\quad
w^{\ep,\eta}\rightarrow w^{\ep} \mbox{ as } \eta\to 0\quad \mbox{
and } w^{\ep,\eta}\leq w^{\ep}.
\end{align*}
Let now $\tau^*\geq 0$ the smallest constant such that
$$w^{\ep,\eta}(x_1-\tau^*)\leq v^{\ep}(x)\quad \mbox{ on }
\overline{\mathcal{R}}.$$ We want to prove that $\tau^*=0$. By the
minimality of $\tau^*$, there exists a point $x^*\in
\overline{\mathcal{R}}$ such that
$w^{\ep,\eta}(x_1^*-\tau^*)=v^{\ep}(x^*)$. If $\tau^*>0$, then
$w^{\ep,\eta}(x_1-\tau^*)<w^{\ep,\eta}(x_1)\leq w^{\ep}(x_1)\leq
v^{\ep}(x)$ on $\partial \mathcal{R}$, and hence, $x^*$ is an
interior point of $\mathcal{R}$.

At this point observe that the gradient of
$w^{\ep,\eta}(x_1-\tau^*)$  does not vanish and
$\mathcal{L}w^{\ep,\eta}(x_1^*-\tau^*)>\beta_{\ep}({w}^{\ep,\eta}(x^*-\tau^*))=
\beta_{\ep}({v}^{\ep}(x^*))=\mathcal{L}v^{\ep}(x^*)$. We also have
$w^{\ep,\eta}(x_1-\tau^*)\leq v^{\ep}(x)$ in $\mathcal{R}$ and
$w^{\ep,\eta}(x_1^*-\tau^*)=v^{\ep}(x^*)$. Then, also $\nabla
w^{\ep,\eta}(x_1^*-\tau^*)=\nabla v^{\ep}(x^*)$.

Let, $$L v=\sum_{i,j=1}^N a_{ij}(\nabla w^{\ep,\eta}(x_1-\tau^*))
v_{x_i x_j}.$$

 Since
 $|\nabla w^{\ep,\eta}(x_1-\tau^*)|>0$ near $x^*$, $L$ is well defined near the point $x^*$ and,
 by condition \eqref{cond}, $L$ is uniformly elliptic.

 Since $\nabla w^{\ep,\eta}(x_1^*-\tau^*)=\nabla v^{\ep}(x^*)$,
  we have that
  $$L w^{\ep,\eta}(x_1^*-\tau^*)=\L w^{\ep,\eta}(x_1^*-\tau^*)>\L \vep(x^*)=\sum_{i,j=1}^N
a_{ij}(\nabla \vep(x^*)) \vep_{x_i x_j}=
   L v^{\ep}(x^*).
   $$

   Moreover, since $\vep$ is a solution
  to
  $$\widetilde{L}v:=\sum_{i,j=1}^N
a_{ij}(\nabla \vep(x)) v_{x_i x_j}=\beta_\ep(v),$$
  $\widetilde{L}$ is uniformly elliptic in a neighborhood of $x^*$ with H\"older continuous
  coefficients and $\beta_\ep(\vep)\in Lip$, there holds that
  $\vep\in C^2$ in a neighborhood of $x^*$.

 Therefore,  we have for some $\delta>0$,
$$
\begin{cases}
L w^{\ep,\eta}(x_1-\tau^*)> L v^{\ep}(x) &\quad \mbox{ in } B_{\delta}(x^*)\\
w^{\ep,\eta}(x_1^*-\tau^*)= v^{\ep}(x^*)\\
w^{\ep,\eta}(x_1-\tau^*)\leq  v^{\ep}(x) &\quad \mbox{ in }
\overline{ \mathcal{R} }.
\end{cases}
$$

But these three statements contradict the strong maximum
principle. Therefore $\tau^*=0$ and thus, $w^{\ep,\eta}\leq
v^{\ep}$ on $\overline{\mathcal{R}}$.

Letting $\eta\to 0$ we obtain the desired result.

\end{proof}

\section{Asymptotic Behavior of Limit Solutions}
\setcounter{equation}{0}

Now we want to prove  --for $g$ satisfying conditions \eqref{cond}
and  \eqref{condderivadagintro2}-- the asymptotic development of
the limiting function $u$. We will obtain this result,  under
suitable assumptions on the function $u$. First we give the
following,
\begin{defi}\label{nondegener} Let $v$ be a continuous nonnegative function
in a domain $\Omega\subset\mathbb{R}^N$. We say that $v$ is
non-degenerate at a point  $x_0\in \Omega\cap\{v=0\}$ if there
exist $c$, $r_0>0$ such that
$$\frac{1}{r^N} \int_{B_r(x_0)} v\, dx\geq c r\quad \mbox{ for } 0<r\leq r_0$$

\medskip

We say that $v$ is uniformly non degenerate in a set
$\Omega'\subset\Omega\cap\{v=0\}$ if the constants $c$ and $r_0$
can be taken independent of the point $x_0\in\Omega'$.
\end{defi}
We have the following,

\begin{teo}\label{teoasim}
Suppose that $g$ satisfies  conditions \eqref{cond} and
\eqref{condderivadagintro2}. Let $u^{\ep_j}$ be a solution to
$(P_{\ep_j})$ in a domain $\Omega\subset \mathbb{R}^N$ such that
$u^{\ep_j}\rightarrow u$ uniformly on compact subsets of $\Omega$
and $\ep_j\to 0$. Let $x_0\in \Omega \cap\partial\{u>0\}$ be such
that $\partial\{u>0\}$ has an inward unit normal $\nu$ in the
measure theoretic sense at $x_0$, and suppose that $u$ is
non-degenerate at $x_0$. Under these assumptions, we have
$$u(x)=\Phi^{-1}(M) \langle x-x_0,\nu\rangle^+ +o(|x-x_0|).$$
\end{teo}
The proof of this theorem makes strong use of the following
result,

\begin{teo}\label{anterior}
Suppose that $g$ satisfies  conditions \eqref{cond} and
\eqref{condderivadagintro2}. Let $u^{\ep_j}$ be a solution to
$(P_{\ep_j})$ in a domain $\Omega\subset\mathbb{R}^N$ such that
$u^{\ep_j}\rightarrow u$ uniformly in compact subsets of $\Omega$
and $\ep_j\to 0$. Let $x_0\in\Omega\cap\partial\{u>0\}$. Then,
$$\limsup_{\stackrel{x\to x_0}{u(x)>0}} |\nabla u(x)|\leq \Phi^{-1}(M).$$
\end{teo}
\begin{proof}
Let
$$
\alpha:=\limsup_{\stackrel{x\to x_0}{u(x)>0}} |\nabla u(x)|.
$$
Since $u\in Lip_{loc}(\Omega)$, $\alpha<\infty$. If, $\alpha=0$ we
are done. So, suppose that $\alpha>0$. By the definition of
$\alpha$ there exists a sequence $z_k\rightarrow x_0$ such that
$$
u(z_k)>0,\quad \quad |\nabla u(z_k)|\rightarrow \alpha.
$$
Let $y_k$ be the nearest point from $z_k$ to $\Omega \cap
\partial\{u>0\}$ and let $d_k = |z_k-y_k|$.

Consider the blow up sequence $u_{d_k}$ with respect to
$B_{d_k}(y_k)$.  This is, $u_{d_k}(x)=\frac 1{d_k}u(y_k+d_k x)$.
Since $u$ is Lipschitz, and $u_{d_k}(0)=0$ for every $k$, there
exists $u_0\in Lip(\R^N)$, such  that (for a subsequence)
$u_{d_k}\to u_0$ uniformly in compact sets of $\R^N$. And we also
have that  $\mathcal{L} u_0=0$ in $\{u_0>0\}$.

Now, set $\bar{z}_k=(z_k-y_k)/d_k\in \partial B_1$. We may assume
that  $\bar{z}_k\to \bar{z}\in \partial B_1$. Take,
$$
\nu_k:=\frac{\nabla u_{d_k}(\bar{z}_k)}{|\nabla
u_{d_k}(\bar{z}_k)|}= \frac{\nabla u({z}_k)}{|\nabla u({z}_k)|}.
$$
Passing to a subsequence and after a rotation we can assume that
$\nu_k\to e_1$. Observe that $B_{2/3}(\bar{z})\subset
B_1(\bar{z}_k)$ for $k$ large, and therefore $u_0$ is an
$\L-$solution there. By interior H\"older gradient estimates (see
\cite{Li1}), we have $\nabla u_{d_k}\to \nabla u_0$ uniformly in
$B_{1/3}(\bar{z})$, and therefore $\nabla u(z_k)\to \nabla
u_0(\bar z)$. Thus, $\nabla u_0(\bar z)=\alpha \,e_1$ and, in
particular, $\partial_{x_1} u_0(\bar{z})=\alpha$.

Next, we claim that $|\nabla u_0|\leq \alpha$ in $\R^N$. In fact,
let $R>1$ and $\delta>0$. Then, there exists, $\tau_0>0$ such that
$|\nabla u(x)|\leq \alpha+\delta$ for any $x\in B_{\tau_0
R}(x_0)$. For $|z_k-x_0|<\tau_0 R/2$ and $d_k<\tau_0/2$ we have,
$B_{d_k R}(z_k)\subset B_{\tau_0 R}(x_0)$ and therefore $|\nabla
u_{d_k}(x)|\leq \alpha +\delta$ in $B_R$ for $k$ large. Passing to
the limit, we obtain $|\nabla u_0|\leq \alpha+\delta$ in $B_R$,
and since $\delta$ and $R$ were arbitrary, the claim holds.

Since $\nabla u_0$ is H\"older continuous in $B_{1/3}(\bar{z})$,
there holds that $\nabla u_0\neq0$ in a neighborhood of $\bar z$.
Thus, by the results in \cite{Li2}, $u_0\in W^{2,2}$ in a ball
$B_r(\bar z)$ for some $r>0$ and, since
$$
\int A(\nabla u_0)\nabla \varphi\,dx=0 \quad\mbox{for every
}\varphi\in C_0^\infty(B_r(\bar z)),
$$
taking $\varphi=\psi_{x_1}$ and integrating by parts we see that,
for $w=\frac{\partial u_0}{\partial x_1}$,
$$
\sum_{i,j=1}^N\int_{B_r(\bar z)} a_{ij}\big(\nabla
u_0(x)\big)w_{x_j}\psi_{x_i}\,dx=0.
$$
This is, $w$ is a solution to the uniformly elliptic equation
$$
\mathcal T w:= \sum_{i=1}^N\frac{\partial}{\partial
x_i}\Big(a_{ij}\big(\nabla u_0(x)\big)w_{x_j}\Big)=0.
$$

Let now $\bar w=\alpha-w$. Then, $\bar w\ge0$ in $B_r(\bar z)$,
$\bar w(\bar z)=0$ and $\mathcal T \bar w=0$ in $B_r(\bar z)$. By
Harnack inequality we conclude that $\bar w\equiv0$. Hence,
$w\equiv\alpha$ in $B_r(\bar z)$.

 Now, since we can repeat this argument around any point where $w=\alpha$,  by a
continuation argument, we have that $w=\alpha$ in $B_1(\bar{z})$.

Therefore, $\nabla u_0=\alpha\, e_1$ and we have, for some $y\in
\R^N$, $u_0(x)=\alpha (x_1-y_1)$ in $B_1(\bar{z})$. Since
$u_0(0)=0$, there holds that $y_1=0$ and $u_0(x)=\alpha x_1$ in
$B_1(\bar{z})$. Finally, since $\L u_0=0$ in $\{u_0>0\}$ by a
continuation argument we have that $u_0(x)=\alpha x_1$ in
$\{x_1\geq 0\}$.

On the other hand, as $u_0\geq 0$, $\L u_0=0$ in $\{u_0>0\}$ and
$u_0=0$ in $\{x_1=0\}$ we have, by Lemma \ref{development11}, that
$$u_0=-\gamma x_1+o(|x|) \quad \mbox{ in } \{x_1<0\}$$
for some $\gamma \geq 0$.

Now, define for $\lambda>0$, $(u_0)_{\lambda}(x)=\frac{1}{\lambda}
u_0(\lambda x)$. There exist a sequence $\lambda_n\to 0$ and
$u_{00}\in Lip(\R^N)$ such that $(u_0)_{\lambda_n}\to u_{00}$
uniformly in compact subsets of $\R^N$. We have $u_{00}(x)=\alpha
x_1^++\gamma x_1^-$.

By Lemma \ref{bolwepj} there exists a sequence $\ep'_j\to 0$ such
that $u^{\ep_j'}$ is a solution to $(P_{\ep'_j})$ and
$u^{\ep'_j}\to u_0$ uniformly on compact subsets of $\R^N$.
Applying a second time Lemma \ref{bolwepj} we find a sequence
$\ep''_j\to 0$ and a solution $u^{\ep''_j}$ to $(P_{\ep''_j})$
converging uniformly in compact subsets of $\R^N$ to $u_{00}$. Now
we can apply Proposition \ref{alphamenorphiM} in the case that
$\gamma=0$ or Proposition \ref{gammapositiva} in the case that
$\gamma>0$, and we conclude that $\alpha\leq \Phi^{-1}(M)$.
\end{proof}

\begin{proof}[Proof of Theorem \ref{teoasim}]
Assume that $x_0=0$, and $\nu=e_1$. Take
$u_{\lambda}(x)=\frac{1}{\lambda} u(\lambda x).$ Let $\rho>0$ such
that $B_{\rho}\subset\subset\Omega$, since $u_{\lambda}\in
Lip(B_{\rho/\lambda})$ uniformly in $\lambda$, $u_{\lambda}(0)=0$,
there exists $\lambda_j\to 0$ and $U\in Lip(\R^ N)$ such that
$u_{\lambda_j}\to U$ uniformly on compact subsets of $\R^N$.  From
Proposition \ref{nose} and Lemma \ref{bolwepj}, $\L u_{\lambda}=0$
in $\{u_{\lambda}>0\}$. Using the fact that $e_1$ is the
 inward normal in the measure theoretic sense, we have, for
fixed $k$,
$$|\{u_{\lambda}>0\}\cap\{x_1<0\}\cap B_k|\to 0 \quad \mbox{ as } \lambda \to 0.$$
Hence, $U=0$ in $\{x_1<0\}$. Moreover, $U$ is non negative in
$\{x_1>0\}$, $\L U=0$ in $\{U>0\}$ and $U$ vanishes in $\{x_1\leq
0\}$. Then, by Lemma \ref{development11} we have that, there
exists $\alpha\geq 0$ such that,
$$U(x)=\alpha x_1^++o(|x|).$$

By Lemma \ref{bolwepj} we can find a sequence $\ep'_j\to 0$ and
solutions $u^{\ep'_j}$ to $(P_{\ep'_j})$ such that $u_{\ep'_j}\to
U$ uniformly on compact subsets of $\R^N$ as $j\to \infty$. Define
$U_{\lambda}(x)=\frac{1}{\lambda} U(\lambda x)$, then
$U_{\lambda}\to \alpha x_1^+$ uniformly on compact subsets of
$\R^N$. Applying again Lemma \ref{bolwepj} we find a second
sequence $\sigma_j\to 0$ and $u^{\sigma_j}$ solution to
$(P_{\sigma_j})$ such that $u^{\sigma_j}\to \alpha x_1^+$
uniformly on compact subsets of $\R^N$ and,
$$\nabla u^{\sigma_j}\to \alpha \chi_{\{x_1>0\}} e_1\quad \mbox{ in } L^{g_0+1}_{loc}(\R^N).$$

Now we proceed as in the proof of Proposition
\ref{alphamenorphiM}. Let $\psi\in C_0^ {\infty}(\R^N)$ and choose
$u_{x_1}^{\sigma_j} \psi$ as test function in the weak formulation
of $\L u^{\sigma_j}=\beta_{\sigma_j}(u^{\sigma_j})$. Then,
$$B_{\sigma_j}(u^{\sigma_j})\to M \chi_{\{x_1>0\}}+\overline{M} \chi_{\{x_1<0\}}
 \quad *-\mbox{weakly in } L^{\infty}$$
with $\overline{M}=0$ or $\overline{M}=M$. Moreover
  $\Phi(\alpha)=M-\overline{M}$.

By the non degeneracy assumption on $u$ we have,
$$\frac{1}{r^N} \int_{B_r} u_{\lambda_j} \, dx \geq cr$$ and then,
$$\frac{1}{r^N} \int_{B_r} U_{\lambda_j} \, dx \geq cr.$$
Therefore $\alpha>0$. So that  we have that $\overline{M}=0$.
Then, $\alpha=\Phi^{-1}(M).$

We have shown that,
$$U(x)=\begin{cases} \Phi^{-1}(M) x_1+o(|x|) &\quad x_1>0\\
0 &\quad x_1\leq 0.
       \end{cases}
$$

By Theorem \ref{anterior}, $|\nabla U|\leq \Phi^{-1}(M)$ in
$\R^N$. As $U=0$ on $\{x_1=0\}$ we have, $U\leq \Phi^{-1}(M)x_1$
in $\{x_1>0\}$.

Since $\L U=0$ in $\{x_1>0\}$, $U=0$ on $\{x_1=0\}$, there holds
that $U\in C^{1,\alpha}(\{x_1\ge0\})$. Thus,
 $|\nabla U(0)|=\Phi^{-1}(M)>0$ so that, near zero, $U$ satisfies a linear
 uniformly elliptic equation in non divergence form
 and the same equation is satisfied by $w=U-\Phi^{-1}(M)x_1$ in $\{x_1>0\}\cap B_r(0)$ for some $r>0$.
 We also have $w\leq 0$ so that  by  Hopf's boundary principle we have that $w=0$ in $\{x_1>0\}\cap B_r(0)$
 and then, by a continuation argument based on the strong maximum principle we deduce that $U(x)=\alpha x_1^+$
 in $\R^N$. The proof is complete.
\end{proof}

\bigskip

Now we prove another result that is needed in order to see that
$u$ is a weak solution according to Definition \ref{weak2}.

\begin{teo}\label{limsup4}
Let $u^{\ep_j}$ be a solution to $(P_{\ep_j})$ in a domain
$\Omega\subset\mathbb{R}^N$ such that $u^{\ep_j}\rightarrow u$
uniformly in compact subsets of $\Omega$ and $\ep_j\to 0$. Let
$x_0\in  \Omega\cap\partial\{u>0\}$ and suppose that $u$ is
non-degenerate at $x_0$. Assume there is  a ball $B$ contained in
$ \{u=0\} $ touching  $ x_0$ then,
\begin{equation}\label{limsup224}
\limsup_{\stackrel{x\to x_0}{u(x)>0}}
\frac{u(x)}{\mbox{dist}(x,B)}= \Phi^{-1}(M).\end{equation}
\end{teo}
\begin{proof}
 Let $\ell$ be the finite limit on the left hand side of \eqref{limsup224}, and $y_k \to x_0$
with $u(y_k)>0$ and
$$\frac{u(y_k)}{d_k}\to \ell, \quad d_k=\mbox{dist}(y_k,B).$$
Consider the blow up sequence $u_k$ with respect to
$B_{d_k}(x_k)$, where $x_k\in\partial B$ are points with
$|x_k-y_k|=d_k$, and choose a subsequence with blow up limit
$u_0$, such that there exists
$$e:=\lim_{k\to\infty} \frac{y_k-x_k}{d_k}.$$

Then, by construction, $u_0(e)=\ell=\ell\langle e,e\rangle$,
$u_0(x)\leq \ell\langle x, e\rangle $ for $\langle x ,
e\rangle\geq 0$, $u_0(x)=0$ for $\langle x, e\rangle\leq 0.$ In
particular, $\nabla u_0(e)=\ell\,e$.

By the non-degeneracy assumption, we have that $\ell>0$. Since
$|\nabla u_0(e)|=\ell>0$ and $\nabla u_0$ is continuous, both
$u_0$ and $\ell\langle x, e\rangle^+$ are solutions of $Lv=0$ in
$\{u_0>0\}\cap \{\langle x,e\rangle\ge 0\}\cap\{|\nabla u_0|>0\}$
where
$$
Lv:=\sum_{i,j=1}^N b_{ij}(\nabla u_0)v_{x_1x_j}
$$
is uniformly elliptic and
$$b_{ij}(p)=\delta_{ij}+\Big(\frac{g'(|p|)|p|}{g(|p|)}-1\Big)
\frac{p_i p_j}{|p|^2}.$$

Now, from the strong maximum principle, we have that  they must
coincide in a neighborhood at the point $e$.

By continuation we have that $u_0=\ell\langle x, e\rangle^+$.
Thus, we have by, Proposition \ref{alphamenorphiM}, that
$\ell=\Phi^{-1}(M)$.

\end{proof}

\section{Regularity of the free boundary}
We can now prove a regularity result for the free boundary of
limits of solutions to \eqref{ecsing},
\begin{teo}\label{regularidadsing} Assume that $g$ satisfies
conditions \eqref{cond} and  \eqref{condderivadagintro2}. Let
$u^{\ep_j}$ be a solution to $(P_{\ep_j})$ in a domain
$\Omega\subset\mathbb{R}^N$ such that $u^{\ep_j}\rightarrow u$
uniformly in compact subsets of $\Omega$ and $\ep_j\to 0$.  Let
$x_0\in \Omega\cap\partial\{u>0\}$ be such that there is an inward
unit normal $\nu$ in the measure theoretic sense at $x_0$. Suppose
that $u$ is uniformly non-degenerate at the free boundary in a
neighborhood of $x_0$ (see Definition \ref{nondegener}). Then,
there exists $r>0$ such that $B_r(x_0)\cap\partial\{u>0\}$ is a
$C^{1,\alpha}$ surface.
\end{teo}
\begin{proof}
By Corollary \ref{gradsing}, Theorem \ref{teoasim}, Theorem
\ref{limsup4} and the nondegeneracy assumption we have that $u$ is
a weak solution in the sense of Definition \ref{weak2}. Therefore
Theorem 9.4 of \cite{MW1} applies, and the result follows.
\end{proof}

\section{Some examples}
In this section we  give some examples in which the nondegeneracy
condition  is satisfied. So that, in these cases
$\partial_{red}\{u>0\}$ is a $C^{1,\alpha}$ surface.

For the case of a limit of minimizers of the functionals
\begin{equation}\label{funbep}
J_{\ep}(v)=\int_\Omega G(|\nabla v|)\,dx+\int_\Omega B_\ep(v)\,dx
\end{equation}
with $B_{\ep}'(s)=\beta_{\ep}(s)$, we wil also prove that
$\H(\partial\{u>0\}\setminus\partial_{red}\{u>0\})=0$.

The uniform non degeneracy condition will follow from the linear
growth out of the free boundary. This is a well known result for
the case of the laplacian. We prove it here for the operator $\L$
(Theorem \ref{nodegfinal}). The proof is based on an iteration
argument that, in the case of the proof for the laplacian, makes
use of the mean value property (see \cite{CS}). We replace it here
by a blow up argument (see Lemma \ref{itera1}).

\bigskip

%
%
%

\begin{lema}\label{itera1}
Let $c_1>1$ and let  $u^{\ep}\in C(\Omega)$, $|\nabla u^{\ep}|\leq
L$ with $\L u^{\ep}=0$ in $\{u^{\ep}>\ep\}$ be such that there
exists $C>0$ so that
 $u^{\ep}(x)\geq C\, dist(x,\partial \{u^{\ep}>\ep\})$ if $\uep(x)>c_1\ep$ and
$d(x)=dist(x,\partial \{u^{\ep}>\ep\})<1/2\, dist(x,\partial
\Omega)$.  Then, there exists $\delta_0>0$,
$\delta_0=\delta_0(c_1,C)$ such that $\forall \ep>0$ and $\forall
x\in \{u^{\ep}>c_1\ep\}$ with $d(x)<1/2\,dist(x,\partial\Omega)$
we have
$$
\sup_{B_{d(x)}(x)} u^{\ep} \geq (1+\delta_0) u^{\ep}(x).
$$
\end{lema}
\begin{proof}
Suppose by contradiction that there exist  sequences $\delta_k\to
0$,  $\ep_k>0$ and $x_k\in \{u^{\ep_k}>c_1 \ep_k\}$ with
$d_k=d(x_k)<1/2\,dist(x_k,\partial\Omega)$ such that
$$
\sup _{ B_{d_k}(x_k)} u^{\ep_k}\leq (1+\delta_k) u^{\ep_k}(x_k).$$

Take $w_k(x)=\displaystyle\frac{u^{\ep_k}(x_k+d_k
x)}{u^{\ep_k}(x_k)}$. Then, $w_k(0)=1$ and
$$
\max_{\overline{B_1}} w_k\leq (1+\delta_k), \quad w_k>0,\quad
\mbox{ and }\quad\mathcal{L}_k w_k=0\quad\mbox{in }B_1,
$$
where $\mathcal{L}_k v=
\mbox{div}\Big(\displaystyle\frac{g_k(|\nabla v|)}{|\nabla v|}
\nabla v \Big)$ with
$g_k(t)=g\Big(\displaystyle\frac{u^{\ep_k}(x_k) t}{d_k}\Big)$.

On the other hand, in $B_2$ we have
$$\|\nabla w_k\|_{L^{\infty}(B_2)}=\|\nabla u^{\ep_k}(x_k+d_k
x)\|_{L^{\infty}(B_2)}\frac{d_k}{u^{\ep_k}(x_k)}\leq
\frac{L}{C}.$$

Then, there exists $\overline{w}\in C(\overline{B}_1)$ such that
$$w_k \to \overline{w}\quad \mbox{uniformly in } \overline{B}_1.$$

Take $0<r<1$ ant let $v_k(x)=(1+\delta_k)-w_k(x)$. Then, since
$g_k$ satisfies \eqref{cond}, by Harnack inequality we have
$$0\leq v_k(x)\leq c(r) v_k(0)\quad \mbox{ for } |x|<r.$$

By passing to the limit we have
$$0\leq 1-\overline{w}\leq c(r) (1-\overline{w}(0))=0.$$

Therefore $\overline{w}=1$ in $B_1$.

Let $y_k\in \partial \{u_k>\ep_k\}$ with $|x_k-y_k|=d_k$. Then,
 if
$z_k=\frac{y_k-x_k}{d_k}$ we have,
$$w_k(z_k)=\frac{\ep_k}{u^{\ep_k}(x_k)}\leq \frac{1}{c_1}$$ and, we may assume that
$z_k\to \bar{z} \in \partial B_1$. Thus,
$1=\overline{w}(\bar{z})\leq \frac{1}{c_1}<1$. This is a
contradiction, and the lemma is proved.
\end{proof}

\begin{teo}\label{nodegfinal} 
Let $c_1>1$, $C,\, L>0$ and $\Omega'\subset\subset\Omega$. There
exist $c_0,\, r_0>0$ such that, if $\uep\in C(\Omega)$ is such
that $\L\uep=0$ in $\{\uep>\ep\}$, $\|\uep\|_{L^\infty(\Omega')},\
\|\nabla\uep\|_{L^\infty(\Omega')}\le L$ and
 $u^{\ep}(x)\geq C\, dist(x,\partial \{u^{\ep}>\ep\})$ if $x\in \{\uep>c_1\ep\}\cap\Omega'$ and
$d(x)=dist(x,\partial \{u^{\ep}>\ep\})<1/2\, dist(x,\partial
\Omega')$ then, if $x_0\in \Omega'\cap \{u^{\ep}>c_1\ep\}$ with
$dist(x_0,\partial \{u^{\ep}>\ep\})< 1/2\, dist(x,\partial
\Omega')$ there holds that,
$$
\sup_{B_r(x_0)} u^{\ep} \geq c_0 r \quad \mbox{ for } 0<r<r_0.
$$
\end{teo}
\begin{proof}
The proof follows as that of Theorem 1.9 in \cite{CS} by using
Lemma \ref{itera1} and the same iteration argument as in that
theorem.
\end{proof}

As a Corollary we get the locally uniform nondegeneracy of
$u=\lim\uep$ if $\uep$ are solutions to $(P_\ep)$ with linear
growth. In fact,
\begin{corol}\label{coro-nondeg}
Let $u^{\ep_j}$ be uniformly bounded solutions to $(P_{\ep_j})$ in
$\Omega$ such that for every $\Omega'\subset\subset\Omega$ there
exist constants $c_1>1$, $C>0$ such that $u^{\ep_j}(x)\geq C\,
dist(x,\partial \{u^{\ep_j}>\ep_j\})$ if $x\in
\{u^{\ep_j}>c_1\ep_j\}\cap\Omega'$ and $d(x)=dist(x,\partial
\{u^{\ep_j}>\ep_j\})<1/2\, dist(x,\partial \Omega')$. Assume
$\uepj\to u$ uniformly on compact subsets of $\Omega$.

Then, there exist constants $c_0,\, r_0$ depending on $c_1,\,C$,
the uniform bound of $\|u^{\ep_j}\|_{L^\infty (\Omega)}$ and
$\Omega'$ such that for every
$x_0\in\Omega'\cap\overline{\{u>0\}}$ such that
$dist(x_0,\partial\{u>0\}) <1/2 dist(x_0,\partial\Omega')$,
$$
\sup_{B_r(x_0)} u \geq c_0 r \quad \mbox{ for } 0<r<r_0.
$$
\end{corol}

\begin{proof}
The proof follows from Theorem \ref{nodegfinal} as in Chapter 1 in
\cite{CS}.
\end{proof}

%

\bigskip

\subsection{Example 1}

Before we  give the first example we need the following,

\begin{defi}\label{minimal}
Let $u^{\ep}$ be a solution to $(P_{\ep})$. We say that $u^{\ep}$
is a minimal solution to $(P_{\ep})$ in $\Omega$ if whenever we
have $h^{\ep}$ a strong supersolution to $(P_{\ep})$ in
$\Omega'\subset\subset\Omega$, i.e.,
$$ h^{\ep}\in W^{1,G}(\Omega)\cap C(\overline{\Omega'}),\quad g(|\nabla h^\ep|)\frac{\nabla h^\ep}{|\nabla h^\ep|}
\in W^{1,1}(\Omega'), \quad \L h^{\ep} \leq \beta_{\ep}(h^{\ep})
\mbox{ in } \Omega',$$ which satisfies
$$h^{\ep}\geq u^{\ep} \mbox{ on } \partial \Omega,$$ then,
$$h^{\ep}\geq u^{\ep} \mbox{ in } \Omega'.$$
\end{defi}

We can prove for minimal solutions, as in Theorem 4.1 in
\cite{BCN}, the following

\begin{lema}
Let  $u^{\ep}$ be minimal solutions to $(P_{\ep})$ in a domain
$\Omega\subset \R^{N}$. For every $\Omega'\subset\subset \Omega$,
there exist $C,\rho$ and $\ep_0$ depending on $N$,
$dist(\Omega',\partial \Omega)$ and the function $\beta$ such
that, if $\ep\leq \ep_0$ and $x\in \Omega'$ then
$$
 u^{\ep}(x)\geq C\,
dist(x, \{u^{\ep}\leq \ep\})
$$
if $dist(x,\{u^{\ep}\leq \ep\})\leq\rho$.
\end{lema}
\begin{proof}
The proof  follows the lines of Theorem 4.1 in \cite{BCN}.
\end{proof}

Then, by Theorems \ref{regularidadsing} and \ref{nodegfinal}, we
have  the following
\begin{teo}\label{ejemplo1}
Assume that $g$ satisfies conditions \eqref{cond} and
\eqref{condderivadagintro2}.  Let $u^{\ep_j}$ be uniformly bounded
minimal solutions to $(P_{\ep_j})$ in a domain
$\Omega\subset\mathbb{R}^N$ such that $u^{\ep_j}\rightarrow u$
uniformly in compact subsets of $\Omega$ as $\ep_j\to 0$.  
Then, $\Omega\cap\partial_{red}\{u>0\}\in C^{1,\alpha}$.
\end{teo}

\bigskip

\subsection{Example 2}

We  consider  solutions of \eqref{ecsing} that are local
minimizers of the functional:
\begin{equation}\label{funbep2}
J_{\ep}(v)=\int_{\Omega} [G(|\nabla v|) + B_{\ep}(v)] \, dx
\end{equation}
where $B_{\ep}'(s)=\beta_{\ep}(s)$. This is, for any
$\Omega'\subset\subset\Omega$,
 $u^{\ep}$ minimizes
 $$\int_{\Omega'} [G(|\nabla v|) + B_{\ep}(v)] \, dx$$ in
$u^{\ep}+W^{1,G}(\Omega')$.

By Theorem \ref{nodegfinal}, in order to prove the nondegeneracy
we only need to prove the   linear growth out of
$\partial\{\uep>\ep\}$.
 The proof follows the
lines of Corollary 1.7 in \cite{CS}.
\begin{lema}
Given $c_1>1$ there exists a constant $C$ such that if $u^{\ep}$
is a local minimizer of $J_{\ep}$ in $B_1$ and
$u^{\ep}(x_0)>c_1\ep$, $x_0\in B_{1/4}$, then
$$ u^{\ep}(x_0)\geq C
dist(x_0, \{u^{\ep}\leq \ep\})$$ if $dist(x_0,\{u^{\ep}\leq
\ep\})\leq 1/4$.
\end{lema}
\begin{proof}
 The proof follows as in Theorem 1.6 in \cite{CS}.
\end{proof}

Therefore, we have that minimizers satisfy the uniform
nondegeneracy condition.

Now, we want to prove that for the limiting function we have that
almost every point of the free boundary belongs to the reduced
free boundary. To this end, we will prove that the limiting
function is a minimizer of the problem treated in \cite{MW1}. We
will follow the steps of Theorem 1.16 in \cite{CS}. We will only
give the details when the proof parts from the one in \cite{CS}.

First we want to estimate the measure of the level sets
$\partial\Omega_{\lambda}$ where
$\Omega_\lambda=\{\uep>\lambda\}$.

For a given set $\mathcal D$ we denote by $\mathcal
N_\delta(\mathcal D)$ the set of points $x$ such that
$dist(x,\mathcal D)<\delta$.
\begin{teo}\label{nmenos1}
Given $c_1>1$ there exist $c_2,c_3>0$ such that if $\lambda\geq
c_1\ep$ and $1/4\geq \delta \geq c_2 \lambda$ then, for $R<1/4$ we
have,
$$|\mathcal{N}_{\delta}(\partial \Omega_{\lambda})\cap B_R)|\leq c_3 \delta
R^{N-1}.$$
\end{teo}

%
%
%

In order to prove this theorem,  we need two  lemmas.

\begin{lema}\label{gradgrad} If $\lambda>\ep$ and $R\leq 3/4$ then,
$$\int_{\{\lambda<u^{\ep}<\delta\}\cap B_R} G(|\nabla u^{\ep}|)\,
dx \leq c \delta R^{N-1}.$$
\end{lema}
\begin{proof}
First, let us prove that  for  $w\in W^{1,G}(B_R)$ such that
$supp\,w\subset\{\uep\ge \lambda\}$ with $\lambda>\ep$, we have
\begin{equation}\label{derivada2}
\int_{B_R} F(|\nabla u^{\ep}|) \nabla  u^{\ep}\nabla w\, dx =
\int_{\partial B_R} w \,F(|\nabla\uep|)\,\frac{\partial
u^{\ep}}{\partial \nu}\, d\H.
\end{equation}

 We follow the ideas in the proof of Lemma \ref{sale}. This is, we
 suppose first that $F(t)\geq c t$ and then, we use an
approximation argument as in that lemma.

If $F(t)\geq ct$ then, by the  estimates of \cite{LU}, we have
that the solutions are in $W^{2,2}(\Omega)$, so equation
\eqref{derivada2} follows by integrating by parts and using that
$\L u^{\ep}=0$ in $\{u^{\ep}>\ep\}$. Finally we use the
approximation argument of Lemma \ref{sale} and the result follows.

Now, let $w=\min\{(u^{\ep}-\lambda)^+,\delta-\lambda\}$. Then,
$w\in W^{1,G}(B_R)$,  $supp\,w\subset\{\uep\ge \lambda\}$ so that,
by \eqref{derivada2} we have,
$$\int_{\{\lambda<u^{\ep}<\delta\} \cap B_R} G(|\nabla u^{\ep}|)\,
dx \leq C \int_{\partial B_R} w\,F(|\nabla\uep|)\, \frac{\partial
u^{\ep}}{\partial \nu}\, d\H\leq C \delta R^{N-1}$$ and the result
follows.
\end{proof}

\begin{lema}\label{lemma7.5}
Given $c_1>1$ there exist $C_1,C_2, c_2>0$ such that, if
$\lambda\geq c_1 \ep$ and $1/8>\delta\geq c_2 \lambda$  we have,
for $R<1/4$,
$$|\mathcal{N}_{\delta}(\partial \Omega_{\lambda})\cap
B_{R}|\leq C_2 \int_{\{\lambda<u^{\ep}<C_1 \delta \}\cap
B_{R+\delta}} G(|\nabla u^{\ep}|) \, dx. $$
\end{lema}
\begin{proof}
First, we cover
$\mathcal{N}_{\delta}(\partial\Omega_{\lambda})\cap B_{R}$ with
balls $B_j=B_{\delta}(x_j)$ with centers $x_j\in
\partial\Omega_{\lambda}\cap B_{R}$ which overlap at most by $n_0$
(with $n_0=n_0(N)$).

We claim that in each  of these balls there exist two subballs
$B_j^1$ and $B_j^2$ with radii $r_j=C\,\delta$ with $C$ to be
fixed below such that, if $v=(u^{\ep}-\lambda)^+$ then,
$$v\geq \frac{c_0}{8}\delta \quad \mbox{ in } B_j^1,\quad\quad v\leq \frac{c_0}{16}\delta \quad \mbox{ in }
B_j^2,$$ where $c_0$ is the constant of nondegeneracy for balls
centered in $B_{1/4}$ with radii at most $1/8$.

In fact, take $B_j^2=B_{r_j}(x_j)$ with $r_j=\frac{c_0}{16
L}\delta$ (here $\|\nabla u^{\ep}\|_{L^{\infty}(B_{3/4})}\le L$).
Observe that, since $u^{\ep}(x_j)=\lambda$ then,  $v(x)\leq L
r_j=\frac{c_0}{16}\delta$ if $x\in B_j^2$.

Let now, $y_j\in \overline{B_{\delta/4}(x_j)}$ such that
$$u^{\ep}(y_j)=\sup_{B_{\delta/4}(x_j)} u^{\ep} \geq c_0
\frac{\delta}{4}.$$

Let $B_j^1=B_{r_j}(y_j)$.  Then, if $x\in B_j^1$,
$$u^{\ep}(x)\geq u^{\ep}(y_j)-L r_j\geq c_0 \frac{\delta}{4} -L
r_j.
 $$

 Thus,
$$u^{\ep}(x)-\lambda \geq c_0 \frac{\delta}{4} -L
r_j -\lambda \geq  (\frac{c_0}{4} - \frac{c_0}{16}-
c_2^{-1})\delta\geq \frac{c_0}{8} \delta$$ if $c_2^{-1}\leq
\frac{c_0}{16}$.

Let $m_j=\pint_{B_j} v$. We claim that in one of the balls
$B_j^1$, $B_j^2$ we must have $|v-m_j|\geq c\delta$ for a certain
constant $c>0$.

Suppose by contradiction that there exist $x_1\in B_j^1$ and
$x_2\in B_j^2$ with
$$|v(x_1)-m_j|< c\delta \quad \quad |v(x_2)-m_j|<
c\delta.
$$

Then,
$$\frac{c_0}{8}\delta-\frac{c_0}{16}\delta\leq
v(x_1)-v(x_2)<2 c\delta$$ which is a contradiction if we take
$c_0/16\geq 2c$.

Therefore, if $k$ is such that  $|B_j^1|=|B_j^2|=k|B_j|$ we have,
by the convexity of $G$ and  Poincare inequality that,
$$\frac{1}{|B_j|}\int_{B_j}
G(|\nabla v|)\, dx\geq G\Big(\frac{1}{|B_j|}\int_{B_j} |\nabla
v|\, dx\Big)\geq G\Big(\frac{C}{|B_j|}\int_{B_j} \frac{|
v-m_j|}{\delta}\, dx\Big)\geq G\Big(\frac{C}{|B_j|} k|B_j|
c\Big).$$

This implies that
$$\int_{B_j}
G(|\nabla v|)\, dx\geq C |B_j|.$$

As
$$B_{R}\cap
\mathcal{N}_{\delta}(\partial\Omega_{\lambda})\subset \bigcup
B_j$$ we have
\begin{align*}
|B_{R}\cap \mathcal{N}_{\delta}(\partial\Omega_{\lambda})|&\leq
\sum |B_j|\leq \frac{1}{C} \sum \int_{B_j} G(|\nabla v|)\, dx\\ &
\leq \frac{n_0}{C} \int_{\bigcup B_j} G(|\nabla v|)\, dx=
\frac{n_0}{C} \int_{\bigcup B_j\cap \{u^{\ep}>\lambda\}} G(|\nabla
u^{\ep} |)\, dx.
\end{align*}
On the other hand, if $x\in B_j$ then $u^{\ep}(x)< C_1 \delta$
where $C_1=c_2^{-1}+L$. Then, as $\bigcup B_j \subset
B_{R+\delta}$, we have
$$|B_{R} \cap
\mathcal{N}_{\delta}(\partial\Omega_{\lambda})|\leq \frac{n_0}{C}
\int_{\{\lambda < u^{\ep} < C_1 \delta \}\cap B_{R+\delta}}
G(|\nabla u^{\ep} |)\, dx.$$
\end{proof}

\medskip

\begin{proof}[Proof of Theorem \ref{nmenos1}] Using Lemmas \ref{gradgrad} and \ref{lemma7.5} we have
 $$|B_{R-\delta} \cap
\mathcal{N}_{\delta}(\partial\Omega_{\lambda})|\leq C_0
\int_{\{\lambda < u^{\ep} < C_1 \delta \}\cap B_{R}} G(|\nabla
u^{\ep} |)\, dx\leq C_0 c C_1 \delta R^{N-1}.$$

 As $|B_R\setminus
B_{R-\delta}|\leq C \delta R^{N-1}$ we obtain the conclusion of
Theorem \ref{nmenos1}.
\end{proof}

Now, we can pass to the limit as $\ep\to 0$.  There exists a
subsequence $u^{{\ep}_k}$ converging, as $\ep_k\to 0$, to   a
function $u_0\in W^{1,G}(\Omega)$ strongly in
$L^{\delta+1}(\Omega)$, weakly in $W^{1,G}(\Omega)$ and uniformly
in every compact subset of $\Omega$.

Let $\Omega'\subset \subset \Omega$, $x_0\in \Omega'\cap
\partial\{u_0>0\}$ and $\rho_0\leq 1/2 dist(\Omega',\partial\Omega)$. Then,
by using the previous results we can prove as in Theorem 1.16 in
\cite{CS} that $u_0$ is a local minimizer of
$$J_0(v):=\int_{B_{\rho}(x_0)} [G(|\nabla v|)+ M \chi_{\{v>0\}}] \, dx.$$

Finally we can apply the results of \cite{MW1} and conclude that
$\H$--almost every  point of the free boundary belongs to the
reduced free boundary. Moreover, by applying the regularity
results for minimizers of $J_0$ from \cite{MW1} (see \cite{M1} for
the regularity of the whole free boundary in dimension 2) we have
the following theorem
\begin{teo}\label{ejemplo2}
Suppose that $g$ satisfies \eqref{cond}. Let $u^{\ep_j}$ be a
local minimizer of  \eqref{funbep2}  in a domain
$\Omega\subset\mathbb{R}^N$ such that $u^{\ep_j}\rightarrow u$
uniformly in compact subsets of $\Omega$ and $\ep_j\to 0$. Then,
$\partial_{red}\{u>0\}$ is a $C^{1,\alpha}$ surface and
$\H(\partial\{u>0\}\setminus\partial_{red}\{u>0\})=0$. In
dimension 2, if there exist $t_0$ and $k$ such that $g(t)\le k\,
t$ for $t\le t_0$ there holds that
 the  whole free boundary is a regular surface.
\end{teo}

\appendix
\renewcommand{\theequation}{\Alph{section}.\arabic{equation}}

\section{Properties of $G$ 
}\label{appAA} \setcounter{equation}{0} The following result is
proved  in \cite{MW1}.
\begin{lema}\label{prop} The function $g$
satisfies the following properties,
\begin{enumerate}
\item[(g1)] $\di \min\{s^{\delta},s^{g_0}\} g(t)\leq g(st)\leq
\max\{s^{\delta},s^{g_0}\} g(t)$
\smallskip
 \item[(g2)] $G$ is convex and $C^2$\item[(g3)]
$\di\frac{t g(t)}{1+g_0}\leq G(t)\leq t g(t) \quad \forall\ t\geq
0.$

\end{enumerate}

\end{lema}

\bigskip

\section{A result on $\mathcal{L}$-solutions  with linear growth}\label{appA}
\setcounter{equation}{0}

In this section we will state some properties of
$\mathcal{L}$-subsolutions.

\begin{lema}\label{psub2} Let $0<r\le1$.
Let $u\in C(\overline{B_r^+})$ be such that $\L u\ge 0$  in
$B_r^+$ and $0\leq u\leq \alpha x_N$ in $B_r^+$, $u\leq \delta_0
\alpha x_N$ on $\partial B_r^+\cap B_{r_0}(\bar{x})$ with
$\bar{x}\in
\partial B^+_r$, $\bar{x}_N>0$ and $0<\delta_0<1$.

Then, there exists $0<\gamma<1$ and $0<\eta\leq 1$, depending only
on $r$ and $N$ such that $$u(x)\leq \gamma \alpha x_N\quad\mbox{in
}B_{\eta r}^+.$$

\end{lema}

\begin{proof}
By the invariance of the equation $\L u\ge 0$ under the rescaling
$\bar{u}(x)=u(rx)/r$ we can suppose that $r=1$.

Let $\psi^{\alpha}$ be a $\L_{\alpha}$-solution in $B_1^+$, with
smooth boundary data, such that
$$
\begin{cases}
\psi^{\alpha}=x_N & \mbox{on } \partial B_1^+\setminus B_{r_0}(\bar{x})\\
\delta_0 x_N\leq \psi^{\alpha} \leq x_N & \mbox{on } \partial
B_1^+\cap B_{r_0}(\bar{x})\\
\psi^{\alpha}=\delta_0 x_N & \mbox{on } \partial B_1^+\cap
B_{r_0/2}(\bar{x}),
\end{cases}
$$
where $\L_{\alpha} v=\mbox{div}
\Big(\displaystyle\frac{g_{\alpha}(|\nabla v|)}{|\nabla v|} \nabla
v\Big)$ and $g_{\alpha}(t)= g(\alpha t)$.

 Therefore, $\L(\alpha \psi^{\alpha})=0$ and, by the comparison
 principle (see \cite{MW1}),
 $u\leq \alpha \psi^{\alpha}$ in $B_1^+$.

 If we see
 that there exist $0<\gamma<1$ and $\eta>0$, independent of
$\alpha$, such that $\psi^{\alpha}\leq \gamma x_N$ in
$B_{\eta}^+$, the result follows.

First, observe that,
\begin{equation}\label{clasealpha}
\delta \leq \frac{g_{\alpha}'(t) t}{g_{\alpha}(t)}\leq g_0.
\end{equation}

Then, by the results in \cite{Li1}, for $0<\rho_0<1$ and some
$0<\beta<1$,
\begin{equation}\label{propphi}
\begin{aligned}
&\psi^{\alpha}\in C^{1,\beta}(\overline{B_{\rho_0}^+})\cap C^\beta
(\overline{B_{1}^+}).
\\& \mbox{The } C^{1,\beta} (\overline{B_{\rho_0}^+})\mbox{ and }C^\beta (\overline{B_{1}^+})\mbox{ norms are bounded by a
constant independent of } \alpha.\\
& \mbox{The constant of the Harnack inequality  is independent of
} \alpha.
\end{aligned}
\end{equation}

If $|\nabla \psi^{\alpha}|\geq \mu>0$ in some open set $U$, we
have  that $\psi^{\alpha}\in W^{2,2}(U)$ and $\psi^\alpha$ is a
solution of the linear uniformly elliptic equation,
\begin{equation}\label{lue}
\mathcal{T_{\alpha}}\psi=\sum_{i,j=1}^N b^{\alpha}_{ij} \psi_{x_i
x_j} =0 \quad \mbox{ in } U,
\end{equation}
where $$b_{ij}^\alpha=\delta_{ij}+\Big(\frac{g'_\alpha(|\nabla
\psi^\alpha|)|\nabla
\psi^\alpha|}{g_\alpha(|\nabla\psi^\alpha|)}-1\Big) \frac{D_i
\psi^\alpha D_j \psi^\alpha}{|\nabla \psi^\alpha|^2},$$ and the
constant of ellipticity depends only on $g_0$ and $\delta$.

%

Now, we divide the proof into several steps.

\bigskip

{\bf Step 1}

 Let $w^{\alpha}=x_N-\psi^{\alpha}$. Then, $w^{\alpha}\in
C^{1,\beta}(\overline{B_{\rho_0}^+})\cap C^\beta
(\overline{B_{1}^+})$ and it is a solution of
$\mathcal{T_{\alpha}}w^{\alpha} =0$ in any open set $U$ where
$|\nabla \psi^{\alpha}|\geq \mu>0$.

On the other hand, as $\psi^{\alpha}\leq x_N$ on $\partial B_1^+$
and both functions are $\L_{\alpha}$-solutions we have, by
comparison, that $\psi^{\alpha}\leq x_N$ in $B_1^+$. Therefore
$w^{\alpha}\geq 0$ in $B_1^+$.


\bigskip

{\bf Step 2}

Let us prove that there exist $\rho$, $\bar{c}$ and $\alpha_0$,
such that  $|\nabla\psi^{\alpha}|\geq \bar{c}$ in $B_{\rho}^+$ if
$0<\alpha\le\alpha_0$.

First, let us see that there exist $c>0$ and $\alpha_1$ such that
\begin{equation}\label{cotanoalpha} \psi^{\alpha}(1/2 e_N)\geq c\quad\mbox{if }0<\alpha\le\alpha_1.
\end{equation}

If not, there exists a sequence  $\alpha_k\to 0$ such that
$\psi^{\alpha_k}(1/2 e_N)\to 0$. Since the constant in Harnack
inequality is independent of $\alpha$ (see \eqref{propphi}), we
have that, $\psi^{\alpha_k}\to 0$ uniformly in compact sets of
$B_1^+$.

On the other hand, using that $\psi^{\alpha}$ are uniformly
bounded in $C^{\beta}(\overline{B_1^+})$, we have that there
exists $\psi\in C^1({B_1^+})\cap C^{\beta}(\overline{B_1^+})$ such
that, for a subsequence, $\psi^ {\alpha_k}\to \psi$ uniformly in
$\overline{B_1^+}$.

Therefore, $\psi=0$ in $\overline{B_1^+}$. But we have that
$\psi=\delta_0 x_N$ on $ B_{r_0/2}(\bar{x})\cap
\partial B_1^+$, which is a contradiction.

\medskip

Now, let $x_1\in \{x_N=0\}\cap B_{1/2}$. Take
$x_0=x_1+\frac{e_N}{4}$. By \eqref{propphi} we have that there
exists a constant $c_1$ independent of $\alpha$ such that,
$\psi^{\alpha}(x)\geq c_1 \psi^{\alpha}(1/2 e_N)$ for any $x\in
\partial B_{1/8}(x_0)$  and therefore, by \eqref{cotanoalpha},
$\psi^{\alpha}\geq \tilde{c}$ on $\partial B_{1/8}(x_0)$.

Take $v=\ep (e^{-\lambda|x-x_0|^2}-e^{-\lambda /16})$, and choose
$\lambda$ such that $\L_\alpha v>0$ in $B_{1/4}(x_0)\setminus
B_{1/8}(x_0)$ and $\ep$ such that $v=\tilde{c}$ on $\partial
B_{1/8}(x_0)$ (observe that, by Lemma 2.9 in \cite{MW1} $\lambda$
and $\ep$ can be chosen independent of $\alpha$).

Since $\psi^{\alpha}\geq 0=v$ on $\partial B_{1/4}(x_0)$ and
$\psi^{\alpha}\geq v$ on $\partial B_{1/8}(x_0)$  we have, by
comparison, that $\psi^{\alpha}\geq v$ in $B_{1/4}(x_0)\setminus
B_{1/8}(x_0)$.

On the other hand $v_{x_N}(x_1)=\ep 2 \lambda (x_0-x_1)_N
e^{-\lambda|x_1-x_0|^2}=\frac{\lambda\ep} 2 e^{-\lambda/16}
=\bar{c}$, and therefore $ \psi^{\alpha}_{x_N}(x_1)\geq \bar{c}$.

As  $\nabla\psi^{\alpha}$ are uniformly H\"older in
$\overline{B_{3/4}^+}$, we have that there exists  $\rho$
independent of $\alpha$ and $x_1$ such that $
\psi^{\alpha}_{x_N}(x)\geq \bar{c}$ in $B_{\rho}^+(x_1)$.

\bigskip

{\bf Step 3}

 Since $|\nabla \psi^{\alpha}|\geq \bar{c}$ in $B_{\rho}^+$, we have that,
  $\mathcal{T}_{\alpha}w^{\alpha}= 0$ there.

  Suppose that
$w^{\alpha}(1/2 \rho e_N)\geq \tilde{c}$, with $\tilde{c}$
independent of $\alpha$. Then, by  Hopf's Principle we have that
there exists $\sigma_1$ depending only on $N$ and the ellipticity
of ${\mathcal T}_\alpha$ such that $w^{\alpha}\geq \sigma_1 x_N$
in $B_{\rho/2}^+$.
  Then, taking
$\gamma=1- \sigma_1$ we obtain the desired result.

\bigskip

{\bf Step 4}

Finally, let as see that the assumption in  Step 3 is satisfied.
This is, let us see that $w^{\alpha}(1/2 \rho e_N )\geq
\tilde{c}>0$ where $\tilde{c}$ is independent of $\alpha$.

Suppose, by contradiction that for a subsequence,
$w^{\alpha_k}(1/2\rho e_N )\to 0$. We know that in $B_{\rho}^+$
$\mathcal{T}_{\alpha} w^{\alpha}=0$. Therefore, applying Harnack
inequality we have  that  $w^{\alpha_k}\to 0$ in $B_\rho^+$.

 On the other hand, since $\psi^{\alpha}\to \psi$ and $\nabla\psi^\alpha\to\nabla\psi$ uniformly in $\overline{B_{\rho_0}^+}$ for
 every $0<\rho_0<1$ there holds that
  $w^{\alpha_k}\to w=x_n-\psi$ in $C^1(\overline{B_{\rho_0}^+})$. Let
$$
\mathcal{A}=\{x\in B_1^+\ /\ {w}=0\},
$$
and suppose that, there exist a point $x_1\in \partial A\cap
B_1^+$. Then, as $w^{\alpha}\geq 0$ we have that ${w}$ attains its
minimum at this point. Therefore $\nabla {w}(x_1)=0$.

Since $\nabla w^{\alpha_k}\to \nabla {w}$ uniformly in a
neighborhood of $x_1$, we have that for some $\tau>0$ independent
of $\alpha_k$, $|\nabla \psi^{\alpha_k}|\geq 1/2$ in
$B_{\tau}(x_1)$. Thus, in this ball,
 $w^{\alpha_k}$ satisfies $\mathcal{T}_{\alpha_k}w^{\alpha_k}=0$.

Now, applying Harnack inequality in $B_{\tau}(x_1)$ and then,
passing to the limit  we obtain that ${w}=0$ in $B_{\tau/2}(x_1)$,
which is a contradiction.

Hence, ${w}=0$ in $\overline{B_1^+}$. But, on the other hand, we
have ${w}=x_N-\delta_0 x_N>0$ on $\partial B_1\cap
B_{r_0/2}(\bar{x})$, which is a contradiction.

%

%
%
\end{proof}

With Lemma \ref{psub2} we can also prove the asymptotic
development of $\L-$ solutions.

\begin{lema}\label{development11}
Let $u$ be  Lipschitz continuous in $\overline{B_1^+}$, $u\geq 0$
in $B_1^+$, $\L u=0$ in $\{u>0\}$ and $u=0$ on $\{x_N=0\}$. Then,
in $B_1^+$ $u$ has the asymptotic development
$$
u(x)=\alpha x_N+ o(|x|),
$$
with $\alpha\ge 0$.
\end{lema}

\begin{proof}
Let
$$
\alpha_j=\inf\{l\ /\ u\leq l x_n \mbox{ in } B_{2^{-j}}^+\}.
$$
Let $\alpha=\lim_{j\to\infty} \alpha_j$.

Given $\ep_0>0$ there exists $j_0$ such that for $j\geq j_0$ we
have $\alpha_j\leq \alpha+\ep_0$. From here, we have $u(x)\leq
(\alpha+\ep_0)x_N$ in $B_{2^{-j}}^+$ so that
$$
u(x)\leq \alpha x_N+o(|x|) \mbox{ in }B_1^+.
$$

Since $u\ge0$, if $\alpha=0$ the result follows. So, let us assume
that $\alpha>0$.

Suppose that $u(x)\neq \alpha x_N+o(|x|)$. Then there exists
$x_k\rightarrow 0$ and $\bar{\delta}>0$ such that
$$
u(x_k)\leq \alpha x_{k,N}-\bar{\delta} |x_k|.
$$
Let $r_k=|x_k|$ and $u_k(x)= r_k^{-1}u(r_k x)$. Then, there exists
$u_0$ such that, for a subsequence that we still call $u_k$,
$u_k\rightarrow u_0$ uniformly in $\overline{B_1^+}$ and
\begin{align*}
& u_k(\bar{x}_k)\leq \alpha \bar{x}_{k,N}-\bar{\delta}\\
& u_k(x)\leq (\alpha+\ep_0) x_N \mbox{ in } B_1^+,
\end{align*}
where $\bar{x}_k=\frac{x_k}{r_k}$, and we can assume that $\bar
x_k\rightarrow x_0$.

In fact, $u(x)\leq (\alpha+\ep_0) x_N$ in $B_{2^{-j_0}}^+$,
therefore $u_k(x)\leq (\alpha+\ep_0) x_N$ in $B_{r_k^{-1}
2^{-j_0}}^+$, and the claim follows if $k$ is big enough so that
$r^{-1}_k 2^{-j_0}\geq 1$.

If we take $\bar{\alpha}=\alpha+\ep_0$ we have
$$
\begin{cases}
\L u_k\geq 0 & \mbox{in }B_1^+\\
u_k=0 & \mbox{on } \{x_N=0\}\\
0\leq u_k \leq \bar{\alpha} x_N & \mbox{on } \partial
B_1^+\\
u_k\leq \delta_0 \bar{\alpha} x_N & \mbox{on } \partial B_1^+\cap
B_{\bar r}(\bar{x}),
\end{cases}
$$
for some $\bar{x} \in \partial B_1^+$, $\bar{x}_N>0$ and some
small $\bar r>0$.

In fact, as the $u_k$'s are continuous with uniform modulus of
continuity, we have
$$
u_k(x_0)\leq \alpha x_{0,N}-\frac{\bar{\delta}}{2}, \mbox{ if }
k\geq \bar{k}.
$$
Moreover there exists $r_0>0$ such that $u_k(x)\leq \alpha
x_N-\frac{\bar{\delta}}{4}$ in $B_{4r_0} (x_0)$. If $x_{0,N}>0$ we
take $\bar{x}=x_0$, if not, we take $\bar{x}\in \partial
B_{3r_0}(x_0)\cap \partial B_1$. Then, $\bar{x}_N>0$ and
$$
u_k(x)\leq \alpha x_N-\frac{\bar{\delta}}{4} \mbox{ in }
B_{r_0}(\bar{x})\subset\subset \{x_N>0\}.
$$
Moreover, there exists $0<\delta_0<1$ such that $\alpha
x_N-\frac{\bar{\delta}}{4}\leq \delta_0 \alpha x_N \leq \delta_0
\bar{\alpha} x_N$ in $B_{ r_0}(\bar{x})$, and the claim follows.

Now, by Lemma \ref{psub2}, there exists $0<\gamma<1$, $\eta>0$
independent of $\ep_0$ and $k$, such that $u_k(x)\leq \gamma
(\alpha+\ep_0) x_N$ in $B_{\eta}^+$. As $\gamma$ and $\eta$ are
independent of $k$ and $\ep_0$, taking $\ep_0\rightarrow 0$, we
have
$$
u_k(x)\leq \gamma \alpha x_N \mbox{ in }B_{\eta}^+.
$$
So that,
$$
u(x)\le \gamma\alpha x_N\mbox{ in }B_{r_k\eta}^+.
$$
Now if $j$ is big enough we have $\gamma \alpha <\alpha_j$ and
$2^{-j}\le r_k\eta$.  But this contradicts the definition of
$\alpha_j$. Therefore,
$$
u(x)=\alpha x_N+o(|x|),
$$
as we wanted to prove.
\end{proof}

\subsection{Example 1}

We  consider  solutions of \eqref{ecsing} that are local
minimizers of the functional \eqref{funbep} i.e: For any
$\Omega'\subset\subset\Omega$,
 $u^{\ep}$ minimizes
 $$\int_{\Omega'} [G(|\nabla v|) + B_{\ep}(v)] \, dx$$ in
$u^{\ep}+W^{1,G}(\Omega')$.

In order to prove the nondegeneracy we will need  a linear growth
result. We will use the following notation: for any $\lambda>0$,
$\Omega_{\lambda}=\{u^{\ep}>\lambda\}$. The proof follows the
lines of Corollary 1.7 in \cite{CS}.
\begin{lema}
Given $c_1>1$ there exist constants $C_1$ and $C_2$ such that if
$u^{\ep}$ is a local minimizer of $J_{\ep}$ in $B_1$ and
$u^{\ep}(x)>c_1\ep$, $x_0\in B_{1/4}$,
$$C_1 dist(x_0,\{u^{\ep}\leq \ep\})\leq u^{\ep}(x_0)\leq C_2
dist(x_0, \{u^{\ep}\leq \ep\})$$ if $dist(x_0,\{u^{\ep}\leq
\ep\})\leq 1/4$.
\end{lema}

Now we can prove the uniform nondegeneracy condition. To this end,
first we have to prove a lemma. For the proof we have to make
different approach of the one in Lemma 1.10 in \cite{CS} since for
our operator we don't have mean value property as in the case of
the Laplacian. Instead, we have to use a blow up argument.
\begin{lema}\label{itera1}
Given $c_1>1$, if $u^{\ep}\in C(\Omega)$, $|\nabla u^{\ep}|\leq
L$,
 $u^{\ep}(x)\geq C dist(x,\partial \{u^{\ep}>\ep\})$ and
$d(x)=dist(x,\partial \{u^{\ep}>\ep\})<1/2 dist(x,\partial
\Omega)$ then there exists $\delta_0>0$ such that $\forall \ep>0$
and $\forall x\in \{u^{\ep}>c_1\ep\}$ we have
$$\sup_{B_d(x)(x)} u^{\ep} \geq (1+\delta_0) u^{\ep}(x) .$$
\end{lema}
\begin{proof}
Suppose by contradiction that there exists a sequence $\delta_k\to
0$,  $\ep_k>0$ and $x_k\in \{u^{\ep_k}>c_1 \ep_k\}$ with
$$\max_{\partial B_{d_k}(x_k)} u^{\ep_k}\leq (1+\delta_k)
u^{\ep_k}(x).$$ Take $w_k(x)=\frac{u^{\ep_k}(x_k+d_k
x)}{u^{\ep_k}(x_k)}$, then $w_k(0)=1$ and
$$\max_{\partial B_1} w_k\leq (1+\delta_k), \quad w_k>0\quad
\mbox{ and }\mathcal{L}_k w_k=0 $$ where $\mathcal{L}_k v=
\mbox{div}\Big(\frac{g_k(|\nabla v|)}{|\nabla v|} \nabla v \Big)$
with $g_k(t)=g\Big(\frac{u^{\ep_k}(x_k) t}{d_k}\Big)$. Therefore,
$$\max_{\overline{B}_1} w_k\leq (1+\delta_k).$$
On the other hand, in $B_2$ we have
$$\|\nabla w_k\|_{L^{\infty}(B_2)}=\|\nabla u^{\ep_k}(x_k+d_k
x)\|_{L^{\infty}(B_2)}\frac{d_k}{u^{\ep_k}(x_k)}\leq
\frac{L}{C}.$$ Let $y_k\in \partial \Omega_{\ep_k}$ with
$|x_k-y_k|=d_k$ then we  have for $x\in B_1$,
$$w_k(x)=\frac{u^{\ep_k}(x_k+x
d_k)}{u^{\ep_k}(x_k)}\leq
\frac{1}{u^{\ep_k}(x_k)}[u^{\ep_k}(y_k)+L 2 d_k]\leq
\frac{1}{c_1}+\frac{2L}{C},$$ therefore
$$0\leq w_k\leq \frac{1}{c_1}+\frac{2L}{C}.$$
Then, there exists $\overline{w}\in C(\overline{B}_1)$ such that
$$w_k \to \overline{w}\quad \mbox{uniformly in } \overline{B}_1.$$
Take $0<r<1$ ant let $v_k(x)=(1+\delta_k)-w_k(x)$ then by Harnack
inequality we have
$$0\leq v_k(x)\leq c(r) v_k(0)\quad \mbox{ for } |x|<r$$
passing to the limit we have
$$0\leq 1-\overline{w}\leq c(r) (1-\overline{w}(0))=0.$$
Therefore $\overline{w}=1$ in $B_1$. On the other hand, if
$z_k=\frac{y_k-x_k}{d_k}$ we have,
$$w_k(z_k)=\frac{\ep_k}{u^{\ep_k}(x_k)}\leq \frac{1}{c_1}$$ and
$z_k\to \bar{z} \in \partial B_1$ then $\overline{w}(\bar{z})\leq
\frac{1}{c_1}<1$ and this is a contradiction since
$\overline{w}\in C(\overline{B}_1)$.

\end{proof}

\begin{teo} Given $c_1>1$, and let $\Omega'\subset\subset \Omega$. Then
there exist $c_0, r_0>0$ such that if $u^{\ep}$ is a local
minimizer of $J_{\ep}$ in $\Omega$, $x_0\in \Omega'\cap
\{u^{\ep}>c_1\ep\}$ and $dist(x_0,\partial \{u^{\ep}>\ep\})<
\frac{r_0}{8}$ we have
$$\sup_{B_r(x_0)} u^{\ep} \geq c_0 r \quad \mbox{ if } 0<r<r_0.$$
\end{teo}
\begin{proof}
The proof follows as in Theorem 1.9 in \cite{CS} using Lemma
\ref{itera1} and the same iteration argument of that Theorem.
\end{proof}

Escribir los otrs lemas

Now we want to prove the following,

\begin{teo}\label{nmenos1}
Given $c_1>1$ there exist $c_2,c_3>0$ such that if $\lambda\geq
c_1\ep$ and $1/4\geq \delta \geq c-2 \lambda$ then for $R<1/4$ we
have,
$$|\mathcal{N}_{\delta}(\partial \Omega)\cap B_R)|\leq c_3 \delta
R^{N-1}.$$
\end{teo}

%
%
%

\begin{lema}\label{gradgrad} If $\lambda>\ep$ and $R\leq 3/4$ then
$$\int_{\{\lambda<u^{\ep}<\delta\}\cap B_R} G(|\nabla u^{\ep}|)\,
dx \leq c \delta R^{N-1}.$$
\end{lema}
\begin{proof}
First let as prove that  for all $w\in W^{1,G}(B_R)$ we have
$$ \int_{B_R} F(|\nabla
u^{\ep}|) \nabla  u^{\ep}\, dx = \int_{\partial B_R} w
\frac{\partial u^{\ep}}{\partial \nu}\, d\H$$

Let $w=\min\{(u^{\ep}-\lambda)^+,\delta-\lambda\}$ then $w\in
W^{1,G}(B_R)$ and
$$\int_{\{\lambda<u^{\ep}<\delta\} \cap B_R} G(|\nabla u^{\ep}|)\,
dx = \int_{\partial B_R} w \frac{\partial u^{\ep}}{\partial \nu}\,
d\H$$

\end{proof}

\begin{lema}
Given $c_1>1$ exist $C_1,C_2, c_2>0$ such that if $\lambda\leq c_1
\ep$, $\delta\geq c_2 \lambda$ and $\delta<1/8$ we have for
$R<1/4$ that
$$|\mathcal{N}_{\delta}(\partial \Omega_{\lambda})\cap
B_{R-\delta}|\leq C_2 \int_{\{\lambda<u^{\ep}<C_1 \delta \}\cap
B_{R+\delta}} G(|\nabla u^{\ep}|) \, dx. $$
\end{lema}
\begin{proof}
First we cover $\mathcal{N}_{\delta}(\partial\Omega_{\lambda})\cap
B_{R-\delta}$ with balls $B_j=B_{\delta}(x_j)$ with centers
$x_j\in \partial\Omega_{\lambda})\cap B_{R}$ which overlaps at
most by $n_0$ (with $n_0=n_0(N)$).

We affirm that in one of these balls there exists two subballs
$B_j^1$ and $B_j^2$ with radios $r_j=O(\delta)$ such that if
$u=(u^{\ep}-\lambda)^+$ then,
$$u\geq \frac{c_0}{8}\delta \quad \mbox{ in } B_j^1,\quad\quad u\leq \frac{c_0}{16}\delta \quad \mbox{ in }
B_j^2,$$ where $c_0$ is the constant of nondegeneracy for balls
centered in $B_{1/4}$ with radios at most $1/8$.

In fact, take $B_j^2=B_{r_j}(x_j)$ with $r_j=\frac{c_0}{16
L}\delta$ (here $\|\nabla u^{\ep}\|_{L^{\infty}(B_{3/4})}$).
Observe that $u^{\ep}(x_j)=\lambda$ then if $x\in B_j^2$ $u(x)\leq
L r_j=\frac{c_0}{16}\delta.$

Let now, $y_j\in \overline{B_{\delta/4}}(x_j)$ such that
$$u^{\ep}(y_j)=\sup_{B_{\delta/4}(x_j)} u^{\ep} \geq c_0
\frac{\delta}{4}.$$ Let $B_j^1=B_{r_j}(y_j)$ if $x\in B_j^1$ then,
$$u^{\ep}(x)\geq u^{\ep}(y_j)-L r_j\geq c_0 \frac{\delta}{4} -L
r_j $$ we have,
$$u^{\ep}(x)-\lambda \geq c_0 \frac{\delta}{4} -L
r_j -\lambda \geq  (\frac{c_0}{4} - \frac{c_0}{16}-
c_2^{-1})\delta\geq \frac{c_0}{8} \delta$$ if $c_2^{-1}\leq
\frac{c_0}{16}$.

Let $m_j=\pint_{B_j} u$. We affirm that in one of the balls
$B_j^1$, $B_j^2$ we must have $|u-m_j|\geq c\delta$. Suppose by
contradiction that there exist $x_1\in B_j^1$ and $x_2\in B_j^2$
with
$$|u(x_1)-m_j|< c\delta \quad \quad |u(x_2)-m_j|<
c\delta$$ then
$$\frac{c_0}{8}\delta-\frac{c_0}{16}\delta\leq
u(x_1)-u(x_2)<2 c\delta$$ which is a contradiction if we take
$c_0/16\geq 2c$.

Therefore, if $|B_j^1|=|B_j^2|=k|B_j|$ we have by the convexity of
$G$ and using Poincare inequality that
$$\frac{1}{|B_j|}\int_{B_j}
G(|\nabla u|)\, dx\geq G\Big(\frac{1}{|B_j|}\int_{B_j} |\nabla
u|\, dx\Big)\geq G\Big(\frac{C}{|B_j|}\int_{B_j} \frac{|
u-m_j|}{\delta}\, dx\Big)\geq G\Big(\frac{1}{|B_j|} k|B_j|
c\Big)$$ which means that
$$\int_{B_j}
G(|\nabla u|)\, dx\geq C |B_j|.$$ As
$$B_{R-\delta}\cap
\mathcal{N}_{\delta}(\partial\Omega_{\lambda})\subset \bigcup
B_j$$ we have
\begin{align*}
|B_{R-\delta}\cap
\mathcal{N}_{\delta}(\partial\Omega_{\lambda})|&\leq \sum
|B_j|\leq \frac{1}{C} \sum \int_{B_j} G(|\nabla u|)\, dx\\ & \leq
\frac{n_0}{C} \int_{\bigcup B_j} G(|\nabla u|)\, dx= \frac{n_0}{C}
\int_{\bigcup B_j\cap \{u^{\ep}>\lambda\}} G(|\nabla u^{\ep} |)\,
dx.
\end{align*}
On the other hand, if $x\in B_j$ then $u^{\ep}(x)< C_1 \delta$
where $C_1=c_2^{-1}+L$. Then, as $\bigcup B_j \subset
B_{R+\delta}$, we have
$$|B_{R-\delta} \cap
\mathcal{N}_{\delta}(\partial\Omega_{\lambda})|\leq \frac{n_0}{C}
\int_{\{\lambda < u^{\ep} < C_1 \delta \}\cap B_{R+\delta}}
G(|\nabla u^{\ep} |)\, dx.$$
\end{proof}

 Now, using Lemma \ref{gradgrad} we have
 $$|B_{R-2\delta} \cap
\mathcal{N}_{\delta}(\partial\Omega_{\lambda})|\leq C_0
\int_{\{\lambda < u^{\ep} < C_1 \delta \}\cap B_{R}} G(|\nabla
u^{\ep} |)\, dx\leq C_0 c C_1 \delta R^{N-1}.$$ As $|B_R\setminus
B_{R-2\delta}|\leq C \delta R^{N-1}$ we obtain the conclusion of
Theorem \ref{nmenos1}.

\appendix
\renewcommand{\theequation}{\Alph{section}.\arabic{equation}}

\section{Properties of $G$ and Orlicz spaces}\label{appAA}
\setcounter{equation}{0} The following results are all include in
\cite{MW1}.
\begin{lema}\label{prop} The function $g$
satisfies the following properties,
\begin{enumerate}
\item[(g1)] $\di \min\{s^{\delta},s^{g_0}\} g(t)\leq g(st)\leq
\max\{s^{\delta},s^{g_0}\} g(t)$
\smallskip
 \item[(g2)] $G$ is convex and $C^2$\item[(g3)]
$\di\frac{t g(t)}{1+g_0}\leq G(t)\leq t g(t) \quad \forall\ t\geq
0.$

\end{enumerate}

\end{lema}

We recall that the functional
$$\|u\|_G=\inf\Big\{k>0:\int_{\Omega} G\Big(\frac{|u(x)|}{k}\Big)\,
dx \leq 1\Big\}$$ is a norm in the Orlicz space $L_{G}(\Omega)$
which is the linear hull of the Orlicz class
$$K_G(\Omega)=\Big\{u \mbox{ measurable }:\ \int_{\Omega}
G(|u|)\, dx<\infty\Big\},$$ observe that this set is convex, since
$G$ is also convex (property (g2)). The Orlicz-Sobolev space
$W^{1,G}(\Omega)$ consists of those functions in $L^{G}(\Omega)$
whose distributional derivatives $\nabla u$ also belong to
$L^G(\Omega)$. And we have that
$\|u\|_{W^{1,G}}=\max\{\|u\|_G,\|\nabla u\|_G\}$ is a norm for
this space.

\section{Blow-up limits}\label{appB}
\setcounter{equation}{0}

Now we  give the definition of blow-up sequence, and we collect
some properties of the limits of these blow-up sequences for
certain classes of functions that are used throughout the paper.

Let $u$ be a function with the following properties,
\begin{enumerate}
\item[(C1)] $u$ is Lipschitz in $\Omega$ with constant $L>0$,
$u\geq 0 \mbox{ in } \Omega$ and $\mathcal{L} u=0 \mbox{ in }
\Omega\cap\{u>0\}$. \item[(C2)] Given $0<\kappa<1$, there exist
two positive constants $C_{\kappa}$ and $r_{\kappa}$ such that for
every ball $B_r(x_0)\subset\Omega$ and $0<r<r_{\kappa}$,
$$
\frac{1}{r}\left(\pint_{B_r(x_0)} u^{\gamma}\, dx
\right)^{1/\gamma}\leq C_{\kappa} \mbox{ implies that } u\equiv 0
\mbox{ in } B_{\kappa r}(x_0).
$$
\end{enumerate}

\begin{defi}
Let $B_{\rho_k}(x_k)\subset\Omega$ be a sequence of balls with
$\rho_k\to 0$, $x_k\to x_0\in \Omega$ and $u(x_k)=0$. Let
$$
u_k(x):=\frac{1}{\rho_k} u(x_k+\rho_k x).
$$
We call $u_k$ a blow-up sequence with respect to
$B_{\rho_k}(x_k)$.
\end{defi}

Since $u$ is locally Lipschitz continuous, there exists a blow-up
limit $u_0:\R^N\to\R$ such that for a subsequence,
\begin{align*}
& u_k\to u_0 \quad \mbox{in} \quad C^\alpha_{\rm loc}(\R^N)\quad
\mbox{for every}\quad 0<\alpha<1,\\
& \nabla u_k\to\nabla u_0\quad *-\mbox{weakly  in}\quad
L^\infty_{\rm loc}(\R^N),
\end{align*}
and $u_0$ is Lipschitz in $\RR^N$ with constant $L$.
\begin{lema}\label{propblowup}
If $u$ satisfies properties {\rm (C1)} and {\rm (C2)}  then,
\begin{enumerate}
\item $u_0\geq 0$ in $\Omega$ and $\mathcal{L} u_0=0$ in
$\{u_0>0\}$

\medskip

\item $\partial\{u_k>0\}\to \partial\{u_0>0\}$ locally in
Hausdorff distance,

\medskip

\item If $K\subset\subset \{u_0=0\}$, then $u_k=0$ in $K$ for big
enough $k$,

\medskip

\item If $K\subset\subset \{u_0>0\}\cup \{u_0=0\}^\circ$, then
$\nabla u_k\rightarrow\nabla u_0$ uniformly in $K$,

\medskip

\item If $x_k\in \partial\{u>0\}$, then $0\in
\partial\{u_0>0\}$
\end{enumerate}
\end{lema}

\begin{proof}
The proof follows as in \cite{FBMW1} and \cite{Le}.
\end{proof}

\section{A result on $\mathcal{L}$-solutions functions with linear growth}\label{appA}
\setcounter{equation}{0}

In this section we will state some properties of
$\mathcal{L}$-subsolutions.

\begin{lema}\label{psub2}
Let $u$ be a $\L-$ solution function in $B_r^+$ such that, $0\leq
u\leq \alpha x_N$ in $B_r^+$, $u\leq \delta_0 \alpha x_N$ on
$\partial B_r^+\cap B_{r_0}(\bar{x})$ with $\bar{x}\in
\partial B^+_r$, $\bar{x}_N>0$ and $0<\delta_0<1$.

Then there exists $0<\gamma<1$ and $0<\ep\leq 1$, depending only
on $r$ and $N$, such that $u(x)\leq \gamma \alpha x_N$ in $B_{\ep
r}^+$.
\end{lema}

\begin{proof}
By the invariance of $\L$- solution under the rescaling
$\bar{u}(x)=u(rx)/r$ we can suppose that $r=1$.

Let $\psi^{\alpha}$ be a $\L_{\alpha}$-solution in $B_1^+$, with
smooth boundary data, such that
$$
\begin{cases}
\psi^{\alpha}=x_N & \mbox{on } \partial B_1^+\setminus B_{r_0}(\bar{x})\\
\delta_0 x_N\leq \psi^{\alpha} \leq x_N & \mbox{on } \partial
B_1^+\cap B_{r_0}(\bar{x})\\
\psi^{\alpha}=\delta_0 x_N & \mbox{on } \partial B_1^+\cap
B_{r_0/2}(\bar{x}),
\end{cases}
$$
where $\L_{\alpha} v=\mbox{div} \Big(\frac{g_{\alpha}(|\nabla
v|)}{|\nabla v|} \nabla v\Big)$ and $g_{\alpha}(t)= g(\alpha t)$.

 Therefore $\L(\alpha \psi^{\alpha})=0$, and by the comparison
 principle(see \cite{MW1})
 $u\leq \alpha \psi^{\alpha}$ in $B_1^+$. If we see
 that there exist $0<\gamma<1$ and $\ep>0$, independent of
$\alpha$, such that $\psi^{\alpha}\leq \gamma x_N$ in $B_{\ep}^+$,
the result follows.

First, observe that,
\begin{equation}\label{clasealpha}
\delta \leq \frac{g_{\alpha}'(t) t}{g_{\alpha}(t)}\leq g_0,
\end{equation}
 then
by \cite{Li1},
\begin{equation}\label{propphi}
\begin{aligned}
&\psi^{\alpha}\in C^{1,\beta}(\overline{B_1^+}) \mbox{ for some
}\beta>0,\\& \mbox{ the } C^{1,\beta} \mbox{ norm is bounded by a
constant independent of } \alpha \mbox{ and }\\& \mbox{ the
constant of the Harnack inequality  is independent of } \alpha.
\end{aligned}
\end{equation}
If $|\nabla \psi^{\alpha}|\geq \mu>0$ in some open set $U$, we
have  that $\psi^{\alpha}\in W^{2,p}(U)$ and is a solution of the
linear uniformly elliptic equation,
\begin{equation}\label{lue}
\mathcal{T_{\alpha}}\psi=\sum_{i,j=1}^N b^{\alpha}_{ij}
\psi^{\alpha}_{x_i x_j} =0 \quad \mbox{ in } U,
\end{equation}
where $b^{\alpha}_{ij}$ was define in \cite{MW1}, and the constant
of ellipticity depends only on $g_0$ and $\delta$.

%
Now, we divide the proof in several steps,

{\bf Step 1}

 Let $w^{\alpha}=x_N-\psi^{\alpha}$ then $w^{\alpha}\in
C^{1,\beta}(\overline{B_1^+})$ and is a solution of
$\mathcal{T_{\alpha}}w^{\alpha} =0$ in any open set $U$ where
$|\nabla \psi^{\alpha}|\geq \mu>0$.

On the other hand, as $\psi^{\alpha}\leq x_N$ in $\partial B_1^+$
and both functions are $\L^{\alpha}$-solutions we have, by
comparison, that $\psi^{\alpha}\leq x_N$ in $B_1^+$. Therefore
$w^{\alpha}\geq 0$ in $B_1^+$.


{\bf Step 2}

Let us prove that, there exist $\rho$ and $\bar{c}$ independent of
$\alpha$, such that  $|\nabla\psi^{\alpha}|\geq \bar{c}$ in
$B_{\rho}^+$.

First, let as see that there exists $c>0$ independent of $\alpha$
such that
\begin{equation}\label{cotanoalpha} \psi^{\alpha}(1/2 e_N)\geq c.
\end{equation}

If not, there exists a sequences of $\alpha_k\to 0$ such that
$\psi^{\alpha_k}(1/2 e_N)\to 0$, but, since the constant in the
Harnack's inequality is independent of $\alpha$ (see
\eqref{propphi}), we have that, $\psi^{\alpha_k}\to 0$ uniformly
in compact sets of $B_1^+$. On the other hand, using that
$\psi^{\alpha}$ are uniformly bounded in
$C^{1,\beta}(\overline{B_1^+})$, we have that there exists
$\psi\in C^1(\overline{B_1^+})$ such that, for a subsequence
$\psi^ {\alpha_k}\to \psi$ uniformly in $\overline{B_1^+}$.
Therefore $\psi=0$ in $\overline{B_1^+}$, but we have that
$\psi=\delta_o x_N$ on $ B_{r_0/2}(\bar{x})\cap
\partial B_1^+$, which is a contradiction.

Let $x_1\in \{x_N=0\}\cap B_{1/2}$, take $x_0=x_1+\frac{e_N}{4}$.
By \eqref{propphi} we have that there exists a constant $c_1$
independent of $\alpha$ such that, $\psi^{\alpha}(x)\geq c_1
\psi^{\alpha}(1/2 e_N)$ for any $x\in \partial B_{1/8}(x_0)$ , and
therefore by \eqref{cotanoalpha} $\psi^{\alpha}\geq \bar{c}$ in
$\partial
B_{1/8}(x_0)$.\\
Take $v=\ep (e^{-\lambda|x-x_0|^2}-e^{-\lambda /16})$, and choose
$\lambda$ such that $\L v>0$ in $B_{1/4}(x_0)\setminus
B_{1/8}(x_0)$ and $\ep$ such that $v=\bar{c}$ on $\partial
B_{1/8}(x_0)$ (observe that, by Lemma 2.9 in \cite{MW1} $\lambda$
and $\ep$ can be chosen independent of $\alpha$). Since
$\psi^{\alpha}\geq 0$ and $\psi^{\alpha}\geq \bar{c}0v$ on
$\partial B_{1/8}(x_0)$ we have by comparison that
$\psi^{\alpha}\geq v$ in $B_{1/4}(x_0)\setminus B_{1/8}(x_0)$. On
the other hand $-v_{x_N}(x_1)=\ep 2 \lambda (x-x_0)_N=\ep 2\lambda
1/4=\bar{c}$, and therefore $ -\psi^{\alpha}_{x_N}(x_1)\geq
\bar{c}$. As the $\psi^{\alpha}$ are uniformly Lipschitz, we have
that there exists a $\rho$ independent of $\alpha$ such that $
-\psi^{\alpha}_{x_N}(x_1)\geq \bar{c}$ in $B_{\rho}^+$ .

{\bf Step 3}

 Since $|\nabla \psi^{\alpha}|\geq \bar{c}$ in $B_{\rho}^+$, we have that,
  $\mathcal{T}_{\alpha}w^{\alpha}= 0$ there. Suppose that\\
$w^{\alpha}(1/2 e_N\rho)\geq \tilde{c}$, with $\tilde{c}$
independent of $\alpha$. Then by the Harnack's inequality we have
that there exists $\sigma_1$ depending on $\beta$ and $N$ such
that, $w^{\alpha}\geq \sigma_1 w^{\alpha}(1/2 e_N \rho)\geq
\sigma_2$ in $B_{\rho/2}^+$, where $\sigma_2$ is a constant
independent of $\alpha$. Therefore $w^{\alpha}\geq \sigma_2
2\rho^{-1} x_N $ in $B_{\rho/2}^+$,
  then taking
$\gamma=1-2\rho^{-1} \sigma_2$ and $\ep=\rho/2$, we obtain the
desired result.

{\bf Step 4}

 Let as see that $w_{\alpha}(1/2 e_N \rho)\geq \tilde{c}>0$
where $\tilde{c}$ is independent of $\alpha$. Suppose, by
contradiction that for a subsequence, $w_{\alpha_k}(1/2 e_N
\rho)\to 0$. We know that in $B_{\rho}^+$ $\mathcal{T}_{\alpha}
w_{\alpha}=0$, then applying Harnack's inequality we have that for
any compact subset $K\subset\subset B$ we have that $w_{\alpha}\to
0$ uniformly in $K$. On the other hand, the $\psi_{\alpha}$ are
uniformly bounded in $C^{1,\beta}(\overline{B_1^+})$, then, there
exists $\bar{w}\in C^1(\overline{B_1^+})$ such that, for a
subsequence $w^{\alpha_k}\to w$ in $C^1(\overline{B_1^+})$. Let
$$
\mathcal{A}=\{x\in B_1^+\ /\ \bar{w}=0\},
$$
and suppose that, there exist a point $x_1\in \partial A\cap
B_1^+$, then as $w^{\alpha}\geq 0$ we have that $\bar{w}$ has a
minimum there, therefore $\nabla \bar{w}(x_1)=0$. As $\nabla
w^{\alpha_k}\to \nabla \bar{w}$ uniformly in $\overline{B_1^+}$,
we have that for some $\tau>0$ independent of $\alpha_k$, $|\nabla
\psi_{\alpha_k}|\geq 1/2$ in $B_{\tau}(x_1)$, then, in this ball,
the $w^{\alpha_k}$ satisfy $\mathcal{T}_{\alpha_k}w^{\alpha_k}=0$.
We can applying Harnack's inequality in $B_{\tau}(x_1)$ and then,
passing to the limit  we obtain that $\bar{w}=0$ in
$B_{\tau/2}(x_1)$, which is a contradiction. Then $\bar{w}=0$ in
$\overline{B_1^+}$, but, on the other hand we have
$\bar{w}=x_N-\delta_0 x_N>0$ on $\partial B_1\cap\partial
B_{r_0/2}(\bar{x})$, which is a contradiction.

%

%
%
\end{proof}

With Lemma \ref{psub2} we can also prove the asymptotic
development of $\L-$ solutions.

\begin{lema}\label{development11}
Let $u$ be  Lipschitz continuous in $\overline{B_1^+}$, $u\geq 0$
in $B_1^+$, $\L$-solution in $\{u>0\}$ and vanishing on $
B_1^+\cap \{x_N=0\}$. Then, in $B_1^+$, $u$ has the asymptotic
development
$$
u(x)=\alpha x_N+ o(|x|),
$$
with $\alpha\ge 0$.
\end{lema}

\begin{proof}
Let
$$
\alpha_j=\inf\{l\ /\ u\leq l x_n \mbox{ in } B_{2^{-j}}^+\}.
$$
Let $\alpha=\lim_{j\to\infty} \alpha_j$.

Given $\ep_0>0$ there exists $j_0$ such that for $j\geq j_0$ we
have $\alpha_j\leq \alpha+\ep_0$. From here, we have $u(x)\leq
(\alpha+\ep_0)x_N$ in $B_{2^{-j}}^+$ so that
$$
u(x)\leq \alpha x_N+o(|x|) \mbox{ in }B_1^+.
$$

If $\alpha=0$ the result follows. Assume that $\alpha>0$ and let
us suppose that $u(x)\neq \alpha x_N+o(|x|)$. Then there exists
$x_k\rightarrow 0$ and $\bar{\delta}>0$ such that
$$
u(x_k)\leq \alpha x_{k,N}-\bar{\delta} |x_k|.
$$
Let $r_k=|x_k|$ and $u_k(x)= r_k^{-1}u(r_k x)$. Then, there exists
$u_0$ such that, for a subsequence that we still call $u_k$,
$u_k\rightarrow u_0$ uniformly in $\overline{B_1^+}$ and
\begin{align*}
& u_k(\bar{x}_k)\leq \alpha \bar{x}_{k,N}-\bar{\delta}\\
& u_k(x)\leq (\alpha+\ep_0) x_N \mbox{ in } B_1^+,
\end{align*}
where $\bar{x}_k=\frac{x_k}{r_k}$, and we can assume that $\bar
x_k\rightarrow x_0$.

In fact, $u(x)\leq (\alpha+\ep_0) x_N$ in $B_{2^{-j_0}}^+$,
therefore $u_k(x)\leq (\alpha+\ep_0) x_N$ in $B_{r_k^{-1}
2^{-j_0}}^+$, and if $k$ is big enough so that $r^{-1}_k
2^{-j_0}\geq 1$.

If we take $\bar{\alpha}=\alpha+\ep_0$ we have
$$
\begin{cases}
\L u_k\geq 0 & \mbox{in }B_1^+\\
u_k=0 & \mbox{on } \{x_N=0\}\\
0\leq u_k \leq \bar{\alpha} x_N & \mbox{on } \partial
B_1^+\\
u_k\leq \delta_0 \bar{\alpha} x_N & \mbox{on } \partial B_1^+\cap
B_{\bar r}(\bar{x}),
\end{cases}
$$
for some $\bar{x} \in \partial B_1^+$, $\bar{x}_N>0$ and some
small $\bar r>0$.

In fact, as $u_k$ are continuous with uniform modulus of
continuity, we have
$$
u_k(x_0)\leq \alpha x_{0,N}-\frac{\bar{\delta}}{2}, \mbox{ if }
k\geq \bar{k}.
$$
Moreover there exists $r_0>0$ such that $u_k(x)\leq \alpha
x_N-\frac{\bar{\delta}}{4}$ in $B_{2r_0} (x_0)$. If $x_{0,N}>0$ we
take $\bar{x}=x_0$, if not, we take $\bar{x}\in B_{2r_0}(x_0)$
with $\bar{x}_N>0$ and
$$
u_k(x)\leq \alpha x_N-\frac{\bar{\delta}}{4}, \mbox{ in }
B_{r_0}(\bar{x})\subset\subset \{x_N>0\}.
$$
As $B_{r_0}(\bar{x})\subset\subset \{x_N>0\}$ there exists
$\delta_0$ such that $\alpha x_N-\frac{\bar{\delta}}{4}\leq
\delta_0 \alpha x_N \leq \delta_0 \bar{\alpha} x_N$ in $B_{\bar
r}(\bar{x})$ for some small $\bar r$, and the claim follows.

Now, by Lemma \ref{psub2}, there exists $0<\gamma<1$, $\ep>0$
independent of $\ep_0$ and $k$, such that $u_k(x)\leq \gamma
(\alpha+\ep_0) x_N$ in $B_{\ep}^+$. As $\gamma$ and $\ep$ are
independent of $k$ and $\ep_0$, taking $\ep_0\rightarrow 0$, we
have
$$
u_k(x)\leq \gamma \alpha x_N \mbox{ in }B_{\ep}^+.
$$
So that,
$$
u(x)\le \gamma\alpha x_N\mbox{ in }B_{r_k\ep}^+.
$$
Now if $j$ is big enough we have $\gamma \alpha <\alpha_j$ and
$2^{-j}\le r_k\ep$.  But this contradicts the definition of
$\alpha_j$. Therefore,
$$
u(x)=\alpha x_N+o(|x|),
$$
as we wanted to prove.
\end{proof}

\bibliographystyle{amsplain}

\def\cprime{$'$}
\providecommand{\bysame}{\leavevmode\hbox
to3em{\hrulefill}\thinspace}
\providecommand{\MR}{\relax\ifhmode\unskip\space\fi MR }
\providecommand{\MRhref}[2]{%
  \href{http://www.ams.org/mathscinet-getitem?mr=#1}{#2}
} \providecommand{\href}[2]{#2}

\end{document}